\documentclass[11pt,reqno]{amsart}

\usepackage{lmodern}
\usepackage[T1]{fontenc}
\usepackage[utf8]{inputenc}
\usepackage[english]{babel}
\usepackage{microtype}
\usepackage{amsmath,amssymb,amsthm,mathtools,float}
\usepackage{enumerate}
\usepackage{xcolor}
\usepackage{tikz}
\usetikzlibrary{arrows.meta}
\usetikzlibrary{decorations.markings} 
\usepackage[pdftex]{hyperref}

\hypersetup{
    unicode=false,          
    pdftoolbar=true,        
   pdfmenubar=true,        
   pdffitwindow=false,     
   pdfstartview={FitH},    
   pdftitle={},    
   pdfauthor={Author},     
    pdfcreator={Creator},   
    pdfproducer={Producer}, 
   pdfkeywords={keyword1} {key2} {key3}, 
   pdfnewwindow=true,      
    colorlinks=true,       
    linkcolor=blue,          
    citecolor=black,        
    filecolor=magenta,      
    urlcolor=black,    
  pdfauthor={Per Alexandersson},
  pdftitle={Sink Polynomials of Line Graphs and Weighted Matching Polynomials}
}

\newtheorem{theorem}{Theorem}[section]
\newtheorem{proposition}[theorem]{Proposition}
\newtheorem{lemma}[theorem]{Lemma}
\newtheorem{corollary}[theorem]{Corollary}
\theoremstyle{definition}
\newtheorem{definition}[theorem]{Definition}
\newtheorem{example}[theorem]{Example}
\newtheorem{remark}[theorem]{Remark}
\newtheorem{question}[theorem]{Question}

\DeclareMathOperator{\sink}{sink}
\DeclareMathOperator{\ao}{ao}

\newcommand{\defin}[1]{%
\relax\ifmmode
\textcolor{blue}{#1}%
\else\textcolor{blue}{\emph{#1}}%
\fi
}
\newcommand{\setV}{V}
\newcommand{\setE}{E}
\newcommand{\matchings}{\mathcal{M}}
\newcommand{\indsets}{\mathcal{I}}
\newcommand{\AO}{\operatorname{AO}}
\newcommand{\Or}{\operatorname{O}}

\usepackage{todonotes}
\presetkeys%
    {todonotes}%
    {inline,backgroundcolor=yellow}{}

\definecolor{pBlue}{RGB}{86,139,190}
\definecolor{pOrange}{RGB}{211,153,80}

\title[Sink Polynomials of Line Graphs]{Sink Polynomials of Line Graphs and Weighted Matching Polynomials}

\author[P.~Alexandersson]{Per Alexandersson}
\address{Department of Mathematics, Stockholm University, SE-106 91 Stockholm, Sweden}
\email{per.w.alexandersson@gmail.com}

\author[L.~Saud]{Leonardo Saud Maia Leite}
\address{Department of Mathematics, KTH Royal Institute of Technology, SE-100 44 Stockholm,
Sweden}
\email{lsml@kth.se}

\date{}

\begin{document}

\raggedbottom

\begin{abstract}
    Chromatic quasisymmetric functions and unicellular LLT polynomials give symmetric-function interpretations of graph orientations counted by sinks. For natural unit interval graphs, we prove that the acyclic sink polynomial is real-rooted for every nonnegative ascent weight.  Equivalently, the elementary coefficients of the chromatic quasisymmetric function, grouped by partition length, form a real-rooted polynomial.  The Carlsson--Mellit relation expresses this polynomial as a plethystic specialization of the corresponding LLT polynomial, and the Hessenberg correspondence gives an interpretation in cohomology.  More generally, the acyclic sink polynomial is real-rooted for chordal claw-free graphs, and the all-orientations sink polynomial is real-rooted for every claw-free graph.  We also study a real-rooted polynomial counting joint minima in local edge orders.  It agrees with the acyclic sink polynomial of the line graph exactly for forests.  An exact formula for unicyclic roots gives a line-graph counterexample to acyclic sink real-rootedness.
\end{abstract}

\maketitle
\tableofcontents

\newpage

\section{Introduction}\label{sec:introduction}

Our motivation comes from the orientation interpretations of chromatic
quasisymmetric functions and unicellular LLT polynomials.  These connect the
coefficients in the elementary basis to statistics on orientations.  We study
the polynomial obtained by grouping those coefficients by partition length,
and ask when its zeros are real.

Let \(H_h\) be the natural unit interval graph on \([n]\) associated with a
Hessenberg function \(h\), and write
\[
X_{H_h}(X;q)=\sum_{\lambda\vdash n}c_\lambda(q)e_\lambda(X)
\]
for its chromatic quasisymmetric function.  The sink theorem of
J.~Shareshian and M.~L.~Wachs~\cite[Theorem~5.3]{ShareshianWachs2016}
identifies its
partition-length specialization with the acyclic sink polynomial:
\[
\sum_{\lambda\vdash n}c_\lambda(q)t^{\ell(\lambda)}
=\sum_{\mathcal O\in\AO(H_h)}
 q^{\operatorname{asc}(\mathcal O)}t^{\sink(\mathcal O)}
=S_{H_h}^{a}(t;q).
\]
Here \(\ell(\lambda)\) counts parts, \(\sink(\mathcal O)\) counts sinks,
and \(\operatorname{asc}(\mathcal O)\) counts edges directed from a smaller
to a larger label.  The Carlsson--Mellit relation connects the same polynomial
to the unicellular LLT polynomial \(\mathcal L_h\).  If \(\Phi_t\) is the
homomorphism defined by \(\Phi_t(e_m)=t\) for \(m\geq1\), then
\[
S_{H_h}^{a}(t;q)
=\Phi_t\bigl(X_{H_h}(X;q)\bigr)
=\Phi_t\!\left(\frac{\mathcal L_h[(q-1)X;q]}{(q-1)^n}\right).
\]
The brackets denote plethystic substitution, and the quotient is interpreted
through the polynomial identity
\(\mathcal L_h[(q-1)X;q]=(q-1)^nX_{H_h}(X;q)\);
see \cite[Lemma~6.1]{AlexanderssonPanova2018} and
Section~\ref{sec:applications} for the conventions.

We prove that \(S_{H_h}^{a}(t;q)\) is real-rooted in \(t\) for every
real \(q\geq0\) (Theorem~\ref{thm:unitIntervalAscent}).  Thus, for each
such \(q\), the elementary coefficients grouped by length are nonnegative,
log-concave and have no internal zeros.  Through the Hessenberg correspondence,
this is also a statement about the homogeneous coefficients of the graded
Frobenius characteristic of regular semisimple Hessenberg cohomology
(Corollary~\ref{cor:hessenbergLength}).

The direct specialization \(\Phi_t(\mathcal L_h(X;1+u))\) counts
half-sinks~\cite[Proposition~5.2]{AlexanderssonPanova2018}: vertices with no
outgoing edge to a larger label.  Their generating polynomials factor into
linear terms read from the area sequence, for both all and acyclic
orientations (Proposition~\ref{prop:lltLengthProduct}).  Ordinary sinks impose
conditions on both smaller and larger neighbors.  Proving real-rootedness for
this coupled statistic is the main problem of the paper.

R.~P.~Stanley's evaluation of the chromatic polynomial at \(-1\) counts
acyclic orientations~\cite{Stanley1973}.  Recording their sinks gives a finer
invariant.  His sink theorem~\cite[Theorem~3.3]{Stanley1995} identifies it
with the elementary coefficients of the chromatic symmetric function grouped
by partition length.  B.-H.~Hwang, W.-S.~Jung, K.-J.~Lee, J.~Oh, and
S.-H.~Yu~\cite{HwangEtAl2021} studied the multivariate generating polynomial
that records the sink set of each acyclic orientation.  We consider its
univariate specialization, in which an orientation with \(k\) sinks contributes
\(t^k\), and compare it with the corresponding sum over all orientations.

The graph-theoretic results extend beyond natural unit interval graphs.
We prove real-rootedness for
acyclic orientations of chordal claw-free graphs
(Theorem~\ref{thm:chordalAcyclic}) and for all orientations of claw-free graphs
(Theorem~\ref{thm:clawFreeAllOrient}).  We also construct a line graph whose
acyclic sink polynomial has nonreal zeros.  Thus the chordal hypothesis in
the first result cannot simply be omitted, even within the class of line graphs.

For a line graph \(L(G)\), a third enumerator explains the forest case.  Choose
a total order on the edges incident to each vertex of the root graph \(G\).
An edge is jointly minimal if it comes first at both endpoints.  We define the
minima polynomial \(N_G(t)\) by counting these families of local orders according
to their number of jointly minimal edges.  It is real-rooted for every \(G\)
(Theorem~\ref{thm:minimaMatching}).
Local edge orders orient each line-graph edge toward the smaller of its two root
edges, and joint minima become sinks.  We prove that
\(N_G(t)=S_{L(G)}^{a}(t)\) holds exactly when \(G\) is a forest
(Theorem~\ref{thm:forestEquality}).

The forest correspondence and the all-orientations theorem extend to prescribed
edge directions: the resulting polynomials are either zero or real-rooted
(Corollaries~\ref{cor:frozenForestRealRooted} and~\ref{cor:frozenClawFree}).
For a connected unicyclic root, we express the acyclic sink polynomial of
its line graph as a difference of weighted matching polynomials
(Theorem~\ref{thm:unicyclicCorrection}).  This formula gives the 16-vertex
counterexample in Example~\ref{ex:unicyclicCounterexample}.

A common indicator expansion relates these enumerators to matching and
independence polynomials.  After translating the variable by one, the
coefficient attached to a prescribed sink set counts orientations in which
every vertex of that set is a sink.  For acyclic orientations, this is the
number of acyclic orientations after deleting the set.  On a chordal graph,
these deletion counts factor into positive vertex weights.  For all
orientations, the directions at a prescribed sink set are forced and every
other edge is free, so product weights arise for every graph.  In the local-order
model, prescribed jointly minimal edges form a matching, and their counts
factor over the endpoints.

These identities allow us to apply the real-rootedness theorem for weighted
matching polynomials of O.~J.~Heilmann and E.~H.~Lieb~\cite{HeilmannLieb1972}
and the independence-polynomial theorem of M.~Chudnovsky and
P.~Seymour~\cite{ChudnovskySeymour2007}, in the weighted form proved by
A.~Engstr\"om~\cite{Engstrom2007}.  A later proof by J.~D.~Leake and
N.~R.~Ryder~\cite{LeakeRyder2019} expresses the latter result through same-phase
stability.

The paper is outlined as follows. Section~\ref{sec:preliminaries} gives the
sink-set expansion and the two weighted real-rootedness theorems. In Section \ref{sec:acyclic-pol}, we study the acyclic sink polynomial for some families of graphs. Section~\ref{sec:chordal} factors the acyclic deletion counts and gives a
second proof through a recurrence for marked sinks.  We derive the minima
matching identity in Section~\ref{sec:minima}, then use the local-order
correspondence to prove the forest and prescribed-direction results in
Section~\ref{sec:forests}. Section~\ref{sec:examples} presents a formula for unicyclic graphs and an example showing that not all claw-free graphs have real-rooted acyclic sink polynomial. Section~\ref{sec:allOrient} treats all orientations, and we show that the sink polynomial of all claw-free graphs is real-rooted. Section~\ref{sec:applications} gives explicit realizations and applications to Hessenberg cohomology and LLT polynomials, and Section~\ref{sec:xDescents} relates sinks to \(X\)-descent sets through Stanley's linear map.  We compare it with fundamental length and Brenti's \(w\)-polynomial, then ask which linear specializations give real-rooted polynomials on these families.

\section{Preliminaries}\label{sec:preliminaries}

All graphs are finite and simple.  For a graph \(G\), we write \(\setV(G)\) and \(\setE(G)\) for its vertex and edge sets, and \(d_G(v)\) for the degree of \(v\in\setV(G)\).  We write \(N_G(v)\) for its neighborhood and \(G[W]\) for the subgraph induced by \(W\subseteq\setV(G)\). The \defin{line graph} \(L(G)\) has
vertex set \(\setE(G)\), and two vertices of \(L(G)\) are adjacent when the corresponding edges of \(G\) have a common endpoint.

An \defin{orientation} of \(G\) assigns a direction to every edge of \(G\). A vertex is a \defin{sink} when every incident edge points toward it.  By convention, an isolated vertex is a sink.  An orientation is \defin{acyclic} if it has no directed cycle.  We write $\defin{\Or(G)}$ for the set of all orientations of $G$, \(\defin{\AO(G)}\) for the set of acyclic orientations, \(\defin{\operatorname{Sink}(\mathcal O)}\) for the sink set of an orientation $\mathcal{O}$, and \(\defin{\sink(\mathcal O)}\coloneqq|\operatorname{Sink}(\mathcal O)|\).

A graph is \defin{claw-free} if it has no induced subgraph isomorphic to \(K_{1,3}\), and \defin{chordal} if every cycle of length at least four has
a chord.  We write \(\defin{\indsets(G)}\) for the independent vertex sets of \(G\), including the empty set.

For a graph \(G\), we define the \defin{acyclic sink polynomial} and the \defin{sink polynomial}, respectively, as
\[
    \defin{S_G^a(t)} \coloneqq \sum_{\mathcal{O}\in\AO(G)} t^{\sink(\mathcal{O})} \quad \mbox{and} \quad \defin{S_G(t)}\coloneqq \sum_{\mathcal{O} \in \Or(G)} t^{\sink(\mathcal{O})}.
\]

\begin{example}
\label{ex:triangle}
    For example, the triangle $K_3$ has six acyclic orientations, each with one sink, and two directed cycles, each with no sink, see Figure \ref{fig:K3-orientations}. Hence
    \[
        S_{K_3}^{a}(t)=6t \quad \mbox{and} \quad S_{K_3}(t)=2+6t.
    \]
    Both polynomials equal \(1\) for the empty graph.
    
    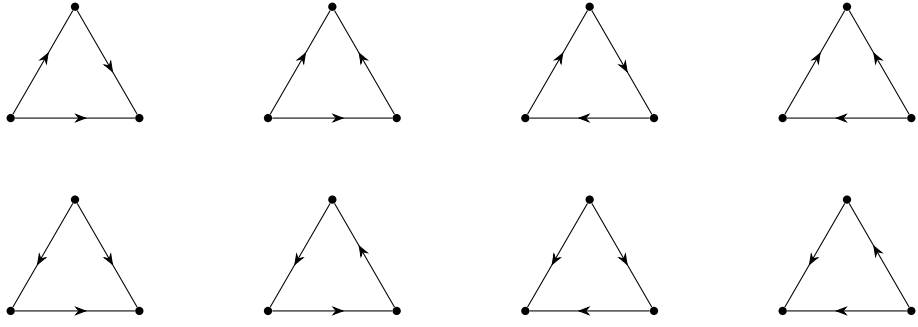
\begin{figure}[!ht]
        \centering
        \begin{tikzpicture}[scale=0.85, vertex/.style={circle,fill=black,inner sep=2.2pt},edge label/.style={fill=white,inner sep=1.5pt,font=\small},>=Stealth]
            \node[scale=0.5, vertex] (v1) at (0,0) {};
            \node[scale=0.5, vertex] (v2) at (2,0) {};
            \node[scale=0.5, vertex] (v3) at (1,1.732) {};
            \draw[postaction={decorate},decoration={markings,mark=at position 0.6 with {\arrow{>}}}] (v1) -- (v2);
            \draw[postaction={decorate},decoration={markings,mark=at position 0.6 with {\arrow{>}}}] (v3) -- (v2);
            \draw[postaction={decorate},decoration={markings,mark=at position 0.6 with {\arrow{>}}}] (v1) -- (v3);
    
            \node[scale=0.5, vertex] (v4) at (4,0) {};
            \node[scale=0.5, vertex] (v5) at (6,0) {};
            \node[scale=0.5, vertex] (v6) at (5,1.732) {};
            \draw[postaction={decorate},decoration={markings,mark=at position 0.6 with {\arrow{>}}}] (v4) -- (v5);
            \draw[postaction={decorate},decoration={markings,mark=at position 0.6 with {\arrow{>}}}] (v5) -- (v6);
            \draw[postaction={decorate},decoration={markings,mark=at position 0.6 with {\arrow{>}}}] (v4) -- (v6);
    
            \node[scale=0.5, vertex] (v7) at (8,0) {};
            \node[scale=0.5, vertex] (v8) at (10,0) {};
            \node[scale=0.5, vertex] (v9) at (9,1.732) {};
            \draw[postaction={decorate},decoration={markings,mark=at position 0.6 with {\arrow{>}}}] (v8) -- (v7);
            \draw[postaction={decorate},decoration={markings,mark=at position 0.6 with {\arrow{>}}}] (v9) -- (v8);
            \draw[postaction={decorate},decoration={markings,mark=at position 0.6 with {\arrow{>}}}] (v7) -- (v9);
    
            \node[scale=0.5, vertex] (v19) at (12,0) {};
            \node[scale=0.5, vertex] (v20) at (14,0) {};
            \node[scale=0.5, vertex] (v21) at (13,1.732) {};
            \draw[postaction={decorate},decoration={markings,mark=at position 0.6 with {\arrow{>}}}] (v20) -- (v19);
            \draw[postaction={decorate},decoration={markings,mark=at position 0.6 with {\arrow{>}}}] (v20) -- (v21);
            \draw[postaction={decorate},decoration={markings,mark=at position 0.6 with {\arrow{>}}}] (v19) -- (v21);
    
            \node[scale=0.5, vertex] (v10) at (0,-3) {};
            \node[scale=0.5, vertex] (v11) at (2,-3) {};
            \node[scale=0.5, vertex] (v12) at (1,-1.268) {};
            \draw[postaction={decorate},decoration={markings,mark=at position 0.6 with {\arrow{>}}}] (v10) -- (v11);
            \draw[postaction={decorate},decoration={markings,mark=at position 0.6 with {\arrow{>}}}] (v12) -- (v11);
            \draw[postaction={decorate},decoration={markings,mark=at position 0.6 with {\arrow{>}}}] (v12) -- (v10);
    
            \node[scale=0.5, vertex] (v13) at (4,-3) {};
            \node[scale=0.5, vertex] (v14) at (6,-3) {};
            \node[scale=0.5, vertex] (v15) at (5,-1.268) {};
            \draw[postaction={decorate},decoration={markings,mark=at position 0.6 with {\arrow{>}}}] (v13) -- (v14);
            \draw[postaction={decorate},decoration={markings,mark=at position 0.6 with {\arrow{>}}}] (v14) -- (v15);
            \draw[postaction={decorate},decoration={markings,mark=at position 0.6 with {\arrow{>}}}] (v15) -- (v13);
    
            \node[scale=0.5, vertex] (v16) at (8,-3) {};
            \node[scale=0.5, vertex] (v17) at (10,-3) {};
            \node[scale=0.5, vertex] (v18) at (9,-1.268) {};
            \draw[postaction={decorate},decoration={markings,mark=at position 0.6 with {\arrow{>}}}] (v17) -- (v16);
            \draw[postaction={decorate},decoration={markings,mark=at position 0.6 with {\arrow{>}}}] (v18) -- (v17);
            \draw[postaction={decorate},decoration={markings,mark=at position 0.6 with {\arrow{>}}}] (v18) -- (v16);
    
            \node[scale=0.5, vertex] (v22) at (12,-3) {};
            \node[scale=0.5, vertex] (v23) at (14,-3) {};
            \node[scale=0.5, vertex] (v24) at (13,-1.268) {};
            \draw[postaction={decorate},decoration={markings,mark=at position 0.6 with {\arrow{>}}}] (v23) -- (v22);
            \draw[postaction={decorate},decoration={markings,mark=at position 0.6 with {\arrow{>}}}] (v23) -- (v24);
            \draw[postaction={decorate},decoration={markings,mark=at position 0.6 with {\arrow{>}}}] (v24) -- (v22);
        \end{tikzpicture}
        \caption{All the orientations of the triangle $K_3$.}
        \label{fig:K3-orientations}
    \end{figure}
\end{example}

The evaluation \(S_G^{a}(1)\) is the number of acyclic orientations of \(G\).  Stanley's evaluation~\cite[Corollary~1.3]{Stanley1973} gives
\[
    S_G^{a}(1)=(-1)^{|\setV(G)|}\chi_G(-1),
\]
where \(\chi_G\) is the chromatic polynomial of $G$.

We now state two real-rootedness theorems for weighted graph polynomials that will be used throughout the whole paper.  Let \(G\) be a graph and let \(\omega_e\geq 0\) for every \(e\in\setE(G)\).  The \defin{weighted matching polynomial} is defined as
\[
    \defin{\mu_G(t;\omega)} \coloneqq \sum_{M\in\matchings(G)}t^{|M|}\prod_{e\in M}\omega_e,
\]
where \(\defin{\matchings(G)}\) is the set of matchings of \(G\), including the empty matching.  We write \(\defin{\setV(M)}\) for the vertices covered by a matching \(M\).

We use the following form of the Heilmann--Lieb theorem:

\begin{theorem}[Heilmann--Lieb~\cite{HeilmannLieb1972}]
\label{thm:heilmannLieb}
    If \(\omega_e\geq 0\) for every \(e\in\setE(G)\), then \(\mu_G(t;\omega)\) has only real nonpositive zeros.
\end{theorem}

We shall also use the weighted extension of the Chudnovsky--Seymour theorem for independence polynomials.  For nonnegative vertex weights \(\eta=\{\eta_v\}_{v\in\setV(G)}\), set
\[
    \defin{\operatorname{I}_G(t;\eta)}\coloneqq \sum_{I\in\indsets(G)}t^{|I|}\prod_{v\in I}\eta_v.
\]

\begin{theorem}[Chudnovsky--Seymour~\cite{ChudnovskySeymour2007}, Engstr\"om~\cite{Engstrom2007}]
\label{thm:weightedClawFree}
    If \(G\) is claw-free and \(\eta_v\geq 0\) for every \(v\in\setV(G)\), then \(\operatorname{I}_G(t;\eta)\) has only real nonpositive zeros.
\end{theorem}

The unweighted theorem is due to Chudnovsky and Seymour. Engstr\"om proved the
weighted form stated here in \cite[Theorem~2.5]{Engstrom2007}.
It also follows from the later same-phase stability
theorem of Leake and Ryder~\cite[Theorem~3.5]{LeakeRyder2019}.

We finish this section by finding an expression for the acyclic sink polynomial of a graph $G$ as a sum over independent sets of $G$. To count orientations by sinks, we first prescribe a set of vertices that must be sinks. Such a set is independent, and removing it leaves an arbitrary acyclic orientation of the remaining graph.  Let \(\defin{\ao(G)}\) denote the number of acyclic orientations of \(G\).

\begin{proposition}[Sink-set expansion]
\label{prop:acyclicIndependentExpansion}
    For every graph \(G\),
    \begin{equation}
    \label{eq:acyclicIndependentExpansion}
        S_G^{a}(t) = \sum_{I\in\indsets(G)} \ao(G-I)(t-1)^{|I|},
    \end{equation}
    where \(G-I\) is the induced subgraph on \(\setV(G)\setminus I\). Consequently, for every graph \(G\),
    \begin{equation}
    \label{eq:acyclicLineMatchingExpansion}
        S_{L(G)}^{a}(t) = \sum_{M\in\matchings(G)} \ao\bigl(L(G-M)\bigr)(t-1)^{|M|},
    \end{equation}
    where \(G-M\) is obtained from $G$ by deleting the edges in \(M\).
\end{proposition}

\begin{proof}
    We expand
    \[
        t^{\sink(\mathcal{O})}=\prod_v \left( 1+(t-1) \mathbf{1}_{v \in \operatorname{Sink}(\mathcal{O})} \right)
    \]
    and sum over all acyclic orientations. Observe that any set of sinks must be independent.  Conversely, if \(I\) is independent, restriction to \(G-I\) is a bijection from acyclic orientations of \(G\) in which every vertex of \(I\) is a sink to acyclic orientations of \(G-I\); the inverse directs every edge incident to \(I\) toward \(I\).  This proves \eqref{eq:acyclicIndependentExpansion}.
    
    Independent sets of \(L(G)\) are the matchings of \(G\), and deleting from
    \(L(G)\) the vertices corresponding to \(M\) gives \(L(G-M)\).  This proves
    \eqref{eq:acyclicLineMatchingExpansion}.
\end{proof}

\section{The acyclic sink polynomial case}
\label{sec:acyclic-pol}

In this section, we prove that the acyclic sink polynomial of all claw-free graphs is real-rooted. We show that $S^a_G(1+t)$ is equal to $\mathrm{I}_G (t; \eta)$ up to a constant for some weights $\eta = \{ \eta_v \}_{v \in V(G)}$, where each $\eta_v$ depend only on $v$, see Remark \ref{rmk:nu-depend-on-v}. We also define the minima polynomial $N_G(t)$, whose coefficients count the number of jointly minimal edges in a graph $G$, and prove that $N_G(t)$ is real-rooted for every graph $G$. In particular, we show that $N_T(t) = S^a_{L(T)}(t)$ for any forest $T$ and, hence, $S^a_{L(T)}(t)$ is real-rooted.

\subsection{Chordal claw-free graphs}
\label{sec:chordal}

We apply Proposition~\ref{prop:acyclicIndependentExpansion} to chordal graphs.  We show that the orientation count after deleting an independent set factors into one weight for each deleted vertex.

A chordal graph $G$ admits a \defin{reverse perfect elimination ordering} (or reverse PEO) \(v_1,\dotsc,v_n\), that is, an ordering of the vertices for which the earlier neighbors of each vertex form a clique. Fix such an ordering and set
\[
    d_j \coloneqq \bigl|N_G(v_j)\cap\{v_1,\dotsc,v_{j-1}\}\bigr|.
\]
Assign vertex \(v_i\) the weight
\begin{equation}\label{eq:chordalSinkWeight}
    \defin{\eta_i} \coloneqq \frac{1}{d_i+1} \prod_{\substack{j>i\\v_iv_j\in\setE(G)}} \frac{d_j}{d_j+1}.
\end{equation}
Every factor in this expression is positive.  Indeed, if \(i<j\) and
\(v_iv_j\) is an edge, then \(v_i\) is an earlier neighbor of \(v_j\), so
\(d_j\geq1\).

\begin{theorem}
\label{thm:chordalAcyclic}
    Let \(G\) be a chordal claw-free graph.  With the notation above,
    \begin{equation}
    \label{eq:chordalSinkIdentity}
        S_G^{a}(1+t)=\left(\prod_{j=1}^n(d_j+1)\right) \operatorname{I}_G(t;\eta).
    \end{equation}
    Consequently, \(S_G^{a}(t)\) has only real nonpositive zeros.
\end{theorem}

\begin{proof}
    We compute \(\ao(G-I)/\ao(G)\) for an independent set \(I\), then substitute in the sink-set expansion.  The reverse perfect elimination ordering gives the
    chromatic factorization
    \[
        \chi_G(t) = \prod_{j=1}^n(t-d_j).
    \]
    To see this, color the vertices in the chosen order.  The earlier neighbors of \(v_j\) form a clique and use \(d_j\) distinct colors, leaving \(t-d_j\)
    choices for \(v_j\).  Stanley's evaluation of the chromatic polynomial gives
    \begin{equation}
    \label{eq:chordalAoCount}
        \ao(G)= (-1)^n\chi_G(-1) = \prod_{j=1}^n(d_j+1).
    \end{equation}
    
    Fix an independent set \(I\) of \(G\), and let
    \[
        \varepsilon_j = \bigl|N_G(v_j)\cap I\cap\{v_1,\dotsc,v_{j-1}\}\bigr|.
    \]
    Since the earlier neighbors of \(v_j\) form a clique, we have \(\varepsilon_j\in\{0,1\}\).  The reverse perfect elimination ordering of $G$ restricts to a reverse perfect elimination ordering of \(G-I\), and the earlier degree of a surviving vertex \(v_j\) in $G - I$ is \(d_j-\varepsilon_j\).  Applying
    \eqref{eq:chordalAoCount} to \(G-I\) gives
    \[
        \ao(G-I) = \prod_{v_j\notin I}(d_j+1-\varepsilon_j).
    \]
    After division by \eqref{eq:chordalAoCount}, we obtain
    \begin{equation}
    \label{eq:chordalDeletionRatio}
        \frac{\ao(G-I)}{\ao(G)} = \prod_{v_i\in I}\frac{1}{d_i+1} \prod_{\substack{v_j\notin I\\\varepsilon_j=1}} \frac{d_j}{d_j+1}.
    \end{equation}
    
    By definition, every index \(j\) in the second product has a unique earlier neighbor \(v_i\in I\), since $\varepsilon_j = 1$. Conversely, if \(v_i\in I\), \(i<j\), and \(v_iv_j\) is an edge, then independence gives \(v_j\notin I\) and \(\varepsilon_j=1\).  Regrouping the second product in \eqref{eq:chordalDeletionRatio} by its unique selected earlier neighbor now gives
    \begin{equation}
    \label{eq:chordalDeletionFactorization}
        \ao(G-I) = \ao(G) \prod_{v_i\in I} \eta_i.
    \end{equation}
    
    Substituting \eqref{eq:chordalDeletionFactorization} in \eqref{eq:acyclicIndependentExpansion}, followed by \eqref{eq:chordalAoCount}, proves \eqref{eq:chordalSinkIdentity}.  Since \(G\) is claw-free and the weights \(\eta_i\) are positive, so Theorem~\ref{thm:weightedClawFree} shows that the polynomial on the right-hand of \eqref{eq:acyclicIndependentExpansion} side has only real nonpositive zeros. Since translation preserves real-rootedness, it follows that \(S_G^{a}(t)\) is real-rooted.
    
    Finally, \(S_G^{a}(t)\) has nonnegative coefficients and hence has no positive real zero.
\end{proof}

\begin{remark}
\label{rmk:nu-depend-on-v}
    The singleton case of \eqref{eq:chordalDeletionFactorization} gives the intrinsic formula
    \[
        \eta_i = \frac{\ao(H-v_i)}{\ao(H)}.
    \]
    Thus the weights $\eta_i$ do not depend on the chosen reverse perfect elimination order, even though \eqref{eq:chordalSinkWeight} computes them from that order.
\end{remark}

\begin{corollary}\label{cor:chordalLineGraphAcyclic}
If \(L(G)\) is chordal, then \(S_{L(G)}^{a}(t)\) has only real
nonpositive zeros.
\end{corollary}

\begin{proof}
Every line graph is claw-free, so the claim follows from
Theorem~\ref{thm:chordalAcyclic}.
\end{proof}

We finish this section by proving the weighted independence identity by inserting vertices along the elimination ordering. We keep track of sinks in an arbitrary marked set: this allows the induction to distinguish sinks adjacent to the inserted vertex from those outside its neighborhood.

For \( W \subseteq\setV(G)\), define the \defin{marked acyclic sink polynomial} to be
\[
    \defin{S_{G,W}^a (t)} \coloneqq \sum_{\mathcal{O}\in\AO(G)} (1+t)^{|\operatorname{Sink}(\mathcal{O})\cap W|}.
\]
For nonnegative weights \(\theta=\{\theta_v\}\), also define the \defin{marked independence polynomial} as
\[
    \defin{\operatorname{I}_{G, W}(t;\theta)} \coloneqq \sum_{\substack{I\in\indsets(G)\\I\subseteq W}} t^{|I|}\prod_{u\in I}\theta_u.
\]

\begin{proposition}
\label{prop:markedChordalModel}
    Let \(G\) have a reverse perfect elimination ordering.  There are \(D>0\) and nonnegative vertex weights \(\theta\) such that
    \begin{equation}
    \label{eq:markedChordalModel}
        S^a_{G,W}(t) = D \cdot \operatorname{I}_{G,W}(t;\theta)
    \end{equation}
    for every \(W \subseteq\setV(G)\) simultaneously.
\end{proposition}

\begin{proof}
    We derive a recurrence for the marked acyclic sink polynomial and choose weights that make the independence polynomial satisfy the same recurrence.  We then
    induct along the reverse perfect elimination ordering.
    
    Let \(v\) be the final vertex, let $H = G - v$, and put
    \[
        K = N_G(v),\qquad d=|K|,\qquad T = W \cap \setV(H).
    \]
    The set \(K\) is a clique.  Fix an acyclic orientation \(\mathcal{O}\) of $H$, and write its induced total ordering on \(K\) as
    \(u_1,\dotsc,u_d\), with \(u_i\to u_j\) for \(i<j\).  Every acyclic
    extension to \(H\) is obtained uniquely by choosing a cut \(0\leq r\leq d\)
    and directing
    \[
    u_i\longrightarrow v\quad(i\leq r),
    \qquad
    v\longrightarrow u_i\quad(i>r).
    \]
    Indeed, a direction pattern that is not a cut creates a directed triangle with \(v\) and two vertices of \(K\).  Conversely, a directed cycle through \(v\) in a cut extension would give a directed path from some \(u_j\), with \(j>r\), to some \(u_i\), with \(i\leq r\). The edge \(u_i\to u_j\) would then complete a directed cycle in $H$. If \(r<d\), the new vertex has an outgoing edge.  Among the vertices of \(K\), only \(u_d\) could have been a sink, and its new edge points toward \(u_d\). Thus the old sink set is unchanged.  If \(r=d\), the vertex \(v\) becomes a sink and every vertex of \(K\) has an outgoing edge toward it.  Summing over the cuts and then over \(\mathcal{O}\) gives
    \begin{equation}
    \label{eq:markedChordalRecurrence}
        S^a_{G,W}(t) = d \cdot S^a_{H,T}(t) +\bigl(1 + t \cdot \mathbf{1}_{v\in W} \bigr) \cdot S^a_{H, T\setminus K}(t).
    \end{equation}
    This also covers \(d=0\), when only the final cut occurs.
    
    For the empty graph, take \(D=1\).  Suppose inductively that \eqref{eq:markedChordalModel} holds on \(H\) with normalization \(D\) and weights \(\theta\).  Put \(c = d+1\), and define weights \(\alpha\) on \(H\) by
    \[
        \alpha_u =
        \begin{cases}
            \dfrac{d}{c}\theta_u,&u\in K,\\[4pt]
            \theta_u,&u\notin K.
        \end{cases}
    \]
    Extend these weights to \(G\) by setting
    \[
        \theta'_v=\frac1c, \qquad \theta'_u=\alpha_u\quad(u\ne v),
    \]
    and put \(D'=cD\).  Since an independent set contains at most one vertex of
    the clique \(K\), splitting according to whether it meets \(K\) gives
    \begin{equation}
    \label{eq:markedCliqueScaling}
        c \cdot \operatorname{I}_{H,T}(t;\alpha) = d \cdot \operatorname{I}_{H,T}(t;\theta) + \operatorname{I}_{H,T\setminus K}(t;\theta).
    \end{equation}
    The weighted independence-polynomial recurrence at \(v\) is
    \begin{equation}
    \label{eq:markedIndependenceRecurrence}
        \operatorname{I}_{H,W}(t;\theta') = \operatorname{I}_{H,T}(t;\alpha) + \mathbf{1}_{v\in W} \cdot \frac{t}{c} \operatorname{I}_{H,T\setminus K}(t;\theta).
    \end{equation}
    The weights \(\alpha\) and \(\theta\) agree outside \(K\), which explains the last term.  Write
    \[
        P = \operatorname{I}_{H,T}(t;\theta), \qquad Q = \operatorname{I}_{H,T\setminus K}(t;\theta).
    \]
    Equations \eqref{eq:markedCliqueScaling} and \eqref{eq:markedIndependenceRecurrence}, together with the induction hypothesis, give
    \begin{align*}
        D' \cdot \operatorname{I}_{G,W}(t;\theta') &=D \cdot \left[ (d \cdot P + Q) + \mathbf{1}_{v\in W} \cdot tQ \right]\\
                                                   &=d \cdot S^a_{H,T}(t) + \bigl( 1 + t \cdot \mathbf{1}_{v\in W} \bigr) \cdot S^a_{H, T \setminus K}(t)\\
                                                   &=S^a_{G,W}(t).
    \end{align*}
    The last equality is \eqref{eq:markedChordalRecurrence}, completing the
    induction for every marked set \(S\).
\end{proof}

\begin{remark}
    Taking \(W=\setV(G)\) in Proposition~\ref{prop:markedChordalModel} gives a second proof of the real-rootedness assertion in Theorem~\ref{thm:chordalAcyclic}. Unwinding the induction recovers the normalization and weights in \eqref{eq:chordalSinkIdentity} and \eqref{eq:chordalSinkWeight}.
\end{remark}

\subsection{The minima polynomial}
\label{sec:minima}

We turn to local orders on the edges of a root graph \(G\).  Their minima
will give the sinks of line-graph orientations in Section~\ref{sec:forests}.
For \(v\in\setV(G)\), let \(\defin{\setE(v)}\) denote the incident edges.
The empty set has one total order.

\begin{definition}
\label{def:minimaPolynomial}
    Choose a total order on \(\setE(v)\) independently for every \(v\in\setV(G)\). A \defin{local edge order} on $G$ is a tuple of total orders in each vertex \(v\in\setV(G)\).  An edge \(uv\in\setE(G)\) is \defin{jointly minimal} in a local edge order if it is the minimum edge at both \(u\) and \(v\) in the same local edge order. The \defin{minima polynomial} of \(G\) is
    \[
        \defin{N_G(t)} \coloneqq \sum_{\text{local edge orders}} t^{\#\{\text{jointly minimal edges}\}}.
    \]
\end{definition}

\begin{example}
\label{ex:K13}
    Let \(G=K_{1,3}\).  The center admits \(3!\) orders, while each leaf admits one.  In every choice, exactly one edge is jointly minimal. Therefore
    \[
        N_{K_{1,3}}(t)=6t.
    \]
    Figure~\ref{fig:localOrders} shows one of these choices and the line-graph orientation it induces.  The jointly minimal edge becomes the unique sink.
        
    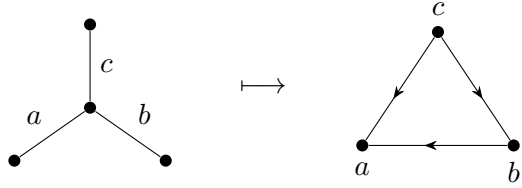
\begin{figure}[!ht]
        \centering
        \begin{tikzpicture}[vertex/.style={circle,fill=black,inner sep=1.6pt}, baseline=(current bounding box.center)]
            \node[vertex] (v) at (0,0) {};
            \node[vertex] (a) at (-1,-0.7) {};
            \node[vertex] (b) at (1,-0.7) {};
            \node[vertex] (c) at (0,1.1) {};
            \draw (v) -- node[above left] {\(a\)} (a);
            \draw (v) -- node[above right] {\(b\)} (b);
            \draw (v) -- node[right] {\(c\)} (c);
        \end{tikzpicture}
        \qquad\(\longmapsto\)\qquad
        \begin{tikzpicture}[vertex/.style={circle,fill=black,inner sep=1.6pt}, >=Stealth,baseline=(current bounding box.center)]
            \node[vertex,label=below:\(a\)] (a) at (-1,0) {};
            \node[vertex,label=below:\(b\)] (b) at (1,0) {};
            \node[vertex,label=above:\(c\)] (c) at (0,1.5) {};
            \draw[postaction={decorate},decoration={markings,mark=at position 0.6 with {\arrow{>}}}] (b) -- (a);
            \draw[postaction={decorate},decoration={markings,mark=at position 0.6 with {\arrow{>}}}] (c) -- (a);
            \draw[postaction={decorate},decoration={markings,mark=at position 0.6 with {\arrow{>}}}] (c) -- (b);
        \end{tikzpicture}
        \caption{The local order \(a<b<c\) at the center of \(K_{1,3}\) directs each edge of its line graph toward the smaller endpoint.  The edge \(a\) is jointly minimal in the star and is a sink in the triangle.}
        \label{fig:localOrders}
    \end{figure}
\end{example}

\begin{theorem}
\label{thm:minimaMatching}
    For every graph \(G\),
    \begin{equation}
    \label{eq:minimaExpansion}
        N_G(t) = \sum_{M\in\matchings(G)}(t-1)^{|M|} \prod_{v\notin\setV(M)}d_G(v)! \prod_{v\in\setV(M)}(d_G(v)-1)!.
    \end{equation}
    Equivalently,
    \begin{equation}
    \label{eq:minimaMatching}
        \frac{N_G(1+t)}{\prod_{v\in\setV(G)}d_G(v)!} = \mu_G\left(t;\left\{\frac{1}{d_G(u)d_G(v)}\right\}_{uv\in\setE(G)} \right).
    \end{equation}
    Consequently, \(N_G(t)\) is real-rooted.
\end{theorem}

\begin{proof}
    For a fixed local edge order \(e\in\setE(G)\), let \(I_e\) be the indicator that \(e\) is jointly minimal.  Then
    \[
        t^{\#\{\text{jointly minimal edges}\}} = \prod_{e\in\setE(G)}\bigl(1+(t-1)I_e\bigr).
    \]
    After expanding the product and summing over local edge orders, the coefficient of \((t-1)^{|F|}\), \(F\subseteq\setE(G)\), is the number of local edge orders in which every edge of \(F\) is jointly minimal.  Two edges sharing a vertex cannot both be minimal there, so only matchings contribute.
    
    Conversely, let \(F=M\) be a matching.  At a vertex in \(v \in \setV(M)\), the edge of \(M\) incident to $v$ is forced to be the local minimum, leaving \((d_G(v)-1)!\) possible total orders for $E(v)$.  At a vertex \(v \notin \setV(M)\), all \(d_G(v)!\) total orders of $E(v)$ are allowed.  The choices at distinct vertices are independent, which gives \eqref{eq:minimaExpansion}.  Notice that this formula also covers isolated vertices: they occur only in the first product and contribute \(0!=1\).
    
    Factoring \(\prod_v d_G(v)!\) from \eqref{eq:minimaExpansion} and replacing \(t\) by \(1+t\) gives \eqref{eq:minimaMatching}.  Every edge has endpoints of positive degree, so all weights in \eqref{eq:minimaMatching} are positive. Theorem~\ref{thm:heilmannLieb} proves the real-rootedness assertion.
\end{proof}

\begin{example}
    Let $G = K_{1,3}$. From Example \ref{ex:K13}, $N_{K_{1,3}}(1+t) = 6 + 6t$. Since $\mathcal{M}(K_{1,3}) = \{ \varnothing, \{a\}, \{b\}, \{c\} \}$ and the degree of the central vertex of $K_{1,3}$ is $3$, while the degrees of the other vertices are $1$, see Figure \ref{fig:localOrders}, we have
    \[
        \mu_{K_{1,3}}\left(t;\left\{\frac{1}{d_G(u)d_G(v)}\right\}_{uv\in\setE(G)} \right) = 1  + 3 \cdot \frac{t}{3} = \frac{N_{K_{1,3}}(1+t)}{3!}.
    \]
\end{example}

In Section \ref{sec:forests}, we prescribe a set of directions $F$ of a graph $G$, and compute its acyclic sink polynomial over all the acyclic orientations of $G$ that contain $F$. It turns out that prescribed directions in a line graph impose comparisons among the root edges at each vertex.  We prepare for this restriction by allowing local partial orders.

\begin{definition}
\label{def:constrainedMinima}
    Choose a partial order $P_v$ on \(\setE(v)\) independently for every \(v\in\setV(G)\). A \defin{local edge order extending $\{P_v\}$} on $G$ is a tuple of total orders extending each $P_v$, \(v\in\setV(G)\). The \defin{constrained minima polynomial} is
    \[
        \defin{N_G(t;\{P_v\})} \coloneqq \sum_{\substack{\text{local edge order}\\\text{extending } \{P_v\}}} t^{\#\{\text{jointly minimal edges}\}}.
    \]
    Let \(\defin{\Lambda_v}\) be the number of linear extensions of \(P_v\), and let \(\defin{\lambda_v(e)}\) be the number of these extensions in which the incident edge \(e\) is the minimum element.
\end{definition}

\begin{theorem}
\label{thm:constrainedMinimaMatching}
    For a graph \(G\) with local partial orders \(\{P_v\}\),
    \begin{align}
    \label{eq:constrainedExpansion}
        N_G(t;\{P_v\}) = {}&\sum_{M\in\matchings(G)}(t-1)^{|M|} \prod_{v\notin\setV(M)}\Lambda_v \prod_{uv\in M}\lambda_u(uv)\lambda_v(uv).   
    \end{align}
    In particular,
    \begin{equation}
    \label{eq:constrainedMatching}
        \frac{N_G(1+t;\{P_v\})}{\prod_{v\in\setV(G)}\Lambda_v} = \mu_G\left(t;\left\{\frac{\lambda_u(uv)\lambda_v(uv)}{\Lambda_u\Lambda_v} \right\}_{uv\in\setE(G)} \right),
    \end{equation}
    and \(N_G(t;\{P_v\})\) is real-rooted.
\end{theorem}

\begin{proof}
    We repeat the indicator expansion from the proof of Theorem~\ref{thm:minimaMatching}. Only matchings contribute. If \(M\in\matchings(G)\), then an endpoint \(v\) of \(uv\in M\) admits \(\lambda_v(uv)\) extensions in which \(uv\) is minimal.  A vertex not covered by \(M\) admits all \(\Lambda_v\) linear extensions.  Independence of the local choices proves \eqref{eq:constrainedExpansion}.
    
    Every finite partial order has a linear extension, so \(\Lambda_v>0\). Factoring \(\prod_v\Lambda_v\) and replacing \(t\) by \(1+t\) gives \eqref{eq:constrainedMatching}. Its edge weights are nonnegative, and Theorem~\ref{thm:heilmannLieb} completes the proof.
\end{proof}

\subsection{Forests}
\label{sec:forests}

We now determine when the local edge orders counted by \(N_G(t)\) produce precisely the acyclic orientations of \(L(G)\). A local edge order on \(G\) determines an orientation of \(L(G)\) as follows: if \(e\) and \(f\) meet at \(v\), we direct the line-graph edge \(\{e,f\}\) toward the smaller edge in the local order at \(v\). This is well-defined because two edges of a simple graph have at most one common endpoint.

\begin{definition}
\label{def:locallyTransitive}
    An orientation of \(L(G)\) is \defin{locally transitive} if, for every \(v\in\setV(G)\), its restriction to the clique on \(\setE(v)\) is transitive.
\end{definition}

\begin{lemma}\label{lem:sinksAreMinima}
    Under this construction, \(e=uv\in\setE(G)\) is a sink of \(L(G)\) if and only if \(e\) is minimal in both \(\setE(u)\) and \(\setE(v)\).
\end{lemma}

\begin{proof}
    Every line-graph edge incident to \(e\) comes from an edge sharing either \(u\)
    or \(v\) with \(e\).  All these edges point toward \(e\) precisely when \(e\)
    is smaller than every other edge at both endpoints.  This is the stated
    condition.
\end{proof}

\begin{proposition}\label{prop:localOrders}
    The local-order construction is a bijection between local edge orders on \(G\) and locally transitive orientations of \(L(G)\).  Under this bijection, joint minima correspond to sinks.  Hence
    \[
        N_G(t) = \sum_{\substack{\mathcal{O} \in \Or(L(G)) \\\text{locally transitive}}} t^{\sink(\mathcal{O})}.
    \]
\end{proposition}

\begin{proof}
    A local edge order orients the clique on \(\setE(v)\) transitively.  Conversely, a transitive orientation of a complete graph determines a unique total order, with each edge directed toward the smaller endpoint.  Every edge of \(L(G)\) belongs to the clique associated with its unique common endpoint in \(G\).
    The constructions are therefore inverse.  The assertion about the statistics follows from Lemma~\ref{lem:sinksAreMinima}.
\end{proof}

The maximal cliques of the line graph of a forest have a particularly simple
form.  We use \defin{block} to mean a maximal 2-connected subgraph, a bridge
with its endpoints, or an isolated vertex.  A \defin{block graph} is a graph
whose blocks are cliques.

\begin{lemma}\label{lem:forestBlockGraph}
Let \(T\) be a forest.  Every block of \(L(T)\) is a clique consisting of the
edges of \(T\) incident to one vertex.  In particular, \(L(T)\) is a block
graph.
\end{lemma}

\begin{proof}
For \(v\in\setV(T)\), let \(C_v\) be the clique on the edges incident to
\(v\).  First consider only vertices with \(d_T(v)\geq2\).  Two of their
cliques meet exactly when their indexing vertices are adjacent in \(T\), and
then their intersection is the single line-graph vertex corresponding to that
root edge.  Thus the intersection graph of these cliques is a forest.
Moreover, a line-graph vertex belongs to at most two of them.

If \(C_u\) and \(C_v\) meet, deleting their common vertex separates the
cliques on the two sides of the indexing edge \(uv\).  Hence no 2-connected
subgraph can cross that intersection.  Each \(C_v\) with at least two vertices
is consequently a block, including the bridge case \(|C_v|=2\).
The only line-graph vertices not covered by these cliques come from isolated
edges of \(T\).  Each is an isolated block and equals the one-element star
clique at either root endpoint.  This proves the assertion.
\end{proof}

We write \(\defin{f\leq_{\mathrm{coeff}}g}\) when every coefficient of \(g-f\) is
nonnegative.

\begin{theorem}\label{thm:forestEquality}
    For every graph \(G\),
    \[
        S_{L(G)}^{a}(t)\leq_{\mathrm{coeff}} N_G(t).
    \]
    Equality holds if and only if \(G\) is a forest.  In particular, if \(T\) is
    a forest, then
    \[
        N_T(t)=S_{L(T)}^{a}(t),
    \]
    and this polynomial is real-rooted.
\end{theorem}

\begin{proof}
Every acyclic orientation of \(L(G)\) restricts to a transitive orientation on
each clique \(\setE(v)\).  Proposition~\ref{prop:localOrders} therefore shows
that \(N_G(t)-S_{L(G)}^{a}(t)\) counts the cyclic locally transitive
orientations of \(L(G)\) by sinks.  This proves the coefficientwise
inequality.

Suppose that \(G\) contains a cycle with consecutive edges
\(e_1,e_2,\dotsc,e_k\), with indices read cyclically.  At the common endpoint $v_i$
of \(e_i\) and \(e_{i+1}\), choose a partial order $P_{v_i}$ for $E(v_i)$ in which
\(e_{i+1}<e_i\), and extend these orderings arbitrarily to a local edge
order.  The resulting orientation of \(L(G)\) contains the directed cycle
\[
    e_1\longrightarrow e_2\longrightarrow\dotsb\longrightarrow e_k \longrightarrow e_1.
\]
It is locally transitive by construction.  Thus the coefficientwise
inequality is strict whenever \(G\) is not a forest.

It remains to prove equality for a forest \(T\). For a vertex $v$, an ordering of $E(v)$ orients the clique \(C_v\) transitively.  We first show
that the resulting orientation of \(L(T)\) is acyclic.  If it contained a
directed cycle, choose one of minimum length.  Its underlying cycle has no
chord, otherwise there would be an orientation of a chord which would produce a shorter directed cycle.
By Lemma~\ref{lem:forestBlockGraph}, every cycle of \(L(T)\) belongs to one
clique \(C_v\).  The chosen directed cycle would therefore lie in a
transitively oriented clique, which is impossible.

Proposition~\ref{prop:localOrders} now gives the polynomial identity, and
Theorem~\ref{thm:minimaMatching} proves real-rootedness.
\end{proof}

The \defin{matching number} of a graph $G$ is defined to be
\[
    \nu(G) \coloneqq \max \{ |M| \colon M \in \mathcal{M}(G)\}.
\]
\begin{proposition}
\label{prop:degreeEvaluations}
    For any graph $G$,
    \[
        \deg N_G(t) = \deg S_{L(G)}^{a}(t) = \deg S_{L(G)}(t) = \nu(G).
    \]
    We also have
    \[
    N_G(1)=\prod_{v\in\setV(G)}d_G(v)!
    \qquad\text{and}\qquad
    S_{L(G)}(1)
    =2^{\sum_{v\in\setV(G)}\binom{d_G(v)}{2}}.
    \]
    If \(G\) is a forest, then also
    \[
    S_{L(G)}^{a}(1)=\prod_{v\in\setV(G)}d_G(v)!.
    \]
\end{proposition}

\begin{proof}
    We start by proving the degree assertion. In fact, jointly minimal edges form a matching, as do sinks in a line graph. This
    bounds all three degrees by \(\nu(G)\).  Now, given a maximum matching \(M\), we can
    make each edge in \(M\) minimal at both endpoints, proving
    \(\deg N_G(t)=\nu(G)\).  We can also orient every line-graph edge incident to a
    vertex in \(M\) toward that vertex, proving the all-orientations degree
    statement.
    
    For the acyclic polynomial, order the vertices of \(L(G)\) with the elements
    of \(M\) first, and direct each edge toward its earlier endpoint.  The
    orientation is acyclic and every element of \(M\) is a sink.  This proves the
    remaining degree statement.
    
    The first evaluation counts all local edge orders.  The second counts all
    orientations of \(L(G)\), and
    \[
    |\setE(L(G))|=\sum_{v\in\setV(G)}\binom{d_G(v)}{2}.
    \]
    The forest evaluation follows from Theorem~\ref{thm:forestEquality}.
\end{proof}

We now restrict the forest correspondence by prescribing some directions.
A \defin{frozen set} \(F\) is a set of
directed versions of edges of a graph \(G\).  We call \(F\) \defin{consistent}
if it contains at most one direction of each edge.  For an arbitrary frozen set \(F\),
define
\[
    \defin{S_G^{a}(t;F)} \coloneqq \sum_{\substack{\mathcal{O} \in \Or(G)\\ \mathcal{O}\supseteq F}} t^{\sink(\mathcal{O})}.
\]
If \(F\) is inconsistent, the indexing set is empty and the polynomial is
zero.

\begin{corollary}\label{cor:frozenForestRealRooted}
Let \(T\) be a forest and let \(F\) be a frozen set in \(L(T)\).  Then
\(S_{L(T)}^{a}(t;F)\) is either zero or real-rooted.
\end{corollary}

\begin{proof}
The claim holds immediately if \(F\) is inconsistent.  Assume that \(F\) is
consistent.  For each \(v\in\setV(T)\), impose the comparison \(e<f\) on
\(\setE(v)\) whenever the line-graph edge \(\{e,f\}\) is frozen toward \(e\).
Take the transitive closure of these comparisons.  If it contains a directed
cycle, no local total order extends the comparisons, and no acyclic
orientation of \(L(T)\) extends \(F\).  The polynomial is then zero.

Otherwise, the closures are partial orders \(P_v\).  The bijection in
Theorem~\ref{thm:forestEquality} restricts to a bijection between acyclic
orientations extending \(F\) and local edge orders extending all \(P_v\).  By
Lemma~\ref{lem:sinksAreMinima},
\[
S_{L(T)}^{a}(t;F)=N_T(t;\{P_v\}).
\]
Theorem~\ref{thm:constrainedMinimaMatching} proves the claim.
\end{proof}

\subsection{Unicyclic graphs}
\label{sec:examples}

A connected graph is \defin{unicyclic} if it has exactly one cycle. In the next theorem, the subgraph inherits the weights from the original graph. We express the acyclic sink polynomial of the line graph $L(G)$ of a unicyclic graph $G$ as a formula of the weighted matching polynomials of $G$ and of $G - E(C)$, where $C$ is the unique cycle of $G$. We use this formula to find a claw-free graph whose acyclic sink polynomial is not real-rooted.

\begin{theorem}\label{thm:unicyclicCorrection}
Let \(G\) be a connected unicyclic graph, and let \(C\) be its unique cycle,
of length \(\ell\). Set
\[
 Z_G=\prod_{v\in\setV(G)}d_G(v)!,
 \qquad
 \omega_{uv}=\frac{1}{d_G(u)d_G(v)}.
\]
Then
\begin{equation}\label{eq:unicyclicCorrection}
S_{L(G)}^{a}(t)
=Z_G\left(
  \mu_G(t-1;\omega)
  -2^{1-\ell}\mu_{G-\setE(C)}(t-1;\omega)
\right).
\end{equation}
\end{theorem}

\begin{proof}
    We evaluate each residual orientation count in \eqref{eq:acyclicLineMatchingExpansion}.  Deleting a matching either breaks the unique cycle or leaves it intact; only the latter case needs a correction. Fix a matching \(M\) of \(G\).  Since \(M\) covers every vertex at most once,
    \begin{equation}
    \label{eq:unicyclicFactorialRatio}
        \prod_{v\in\setV(G)}d_{G-M}(v)! = Z_G\prod_{uv\in M}\frac{1}{d_G(u)d_G(v)}.
    \end{equation}
    If \(M\) contains a cycle edge, then \(G-M\) is a forest. By Proposition~\ref{prop:degreeEvaluations} and \eqref{eq:acyclicLineMatchingExpansion}, the right-hand side of \eqref{eq:unicyclicFactorialRatio} equals \(\ao(L(G-M))\).
    
    Suppose instead that \(M\cap\setE(C)=\varnothing\). The graph \(G-M\) has one
    unicyclic component and some forest components. A family of local orders on
    \(G-M\) gives a cyclic orientation of its line graph exactly when the edges of
    \(C\) are coherently directed. Indeed, a shortest directed cycle is chordless.
    It cannot lie in one root star, because every root star is transitively
    oriented.  To identify the remaining possibilities, let
    \(e_1,\dotsc,e_k\) be the root edges corresponding to a chordless line-graph
    cycle.  If \(k=3\), three pairwise incident edges either share a common
    endpoint or form a root triangle.  If \(k\geq4\), nonconsecutive edges
    \(e_i\) are disjoint.  Their successive common endpoints are therefore
    distinct, and \(e_1,\dotsc,e_k\) form a root cycle in that order.  Thus a
    shortest directed cycle not contained in a star must come from \(C\),
    including when \(C\) is a triangle.  Conversely, coherent directions around
    these cycle edges give a directed cycle in the line graph.
    
    For either coherent direction, one comparison between the two cycle edges is
    prescribed at each of the \(\ell\) cycle vertices. Exactly half of the local
    orders at each such vertex satisfy the comparison. Thus the two coherent
    directions account for the fraction \(2^{1-\ell}\) of all local-order
    families. Consequently,
    \[
    \ao(L(G-M))
    =Z_G\prod_{e\in M}\omega_e
     \left(1-2^{1-\ell}\right)
    \]
    when \(M\cap\setE(C)=\varnothing\), while the factor in parentheses is absent
    when \(M\) meets \(C\). Substitution in
    \eqref{eq:acyclicLineMatchingExpansion} proves
    \eqref{eq:unicyclicCorrection}.
\end{proof}

\begin{example}
\label{ex:unicyclicCounterexample}
Let \(G\) be obtained from a cycle with vertices
\(v_0,v_1,\dotsc,v_{11}\), in cyclic order, by attaching one path of length
two at each of \(v_0\) and \(v_4\). Thus \(G\) is a connected simple
unicyclic graph on 16 vertices, with maximum degree three; see
Figure~\ref{fig:unicyclicCounterexample}.  The two terms in
\eqref{eq:unicyclicCorrection}, divided by 8, are respectively
\[
(t+1)^2\left(
 t^6+90t^5+975t^4+2476t^3+975t^2+90t+1
\right)
\]
and
\[
(1+2t)^2.
\]
Thus direct evaluation of \eqref{eq:unicyclicCorrection} gives
\[
\begin{aligned}
S_{L(G)}^{a}(t)
={}&8t\bigl(
 88+1152t+4516t^2+6902t^3\\
&\hspace{34mm}{}+4516t^4+1156t^5+92t^6+t^7
\bigr),
\end{aligned}
\]
which is a polynomial that has a nonreal conjugate pair of roots. In particular, the acyclic sink
polynomial of a line graph, and hence of a claw-free graph, need not be
real-rooted.
\end{example}

\begin{figure}[!ht]
\centering
\begin{tikzpicture}[vertex/.style={circle,fill=black,inner sep=1.6pt}]
  \foreach \i in {0,...,11}
    \node[vertex] (v\i) at ({90-30*\i}:1.35) {};
  \foreach \i [evaluate=\i as \j using {int(mod(\i+1,12))}] in {0,...,11}
    \draw (v\i) -- (v\j);
  \node[vertex] (a0) at (0,2.05) {};
  \node[vertex] (b0) at (0,2.75) {};
  \draw (v0) -- (a0) -- (b0);
  \node[vertex] (a4) at (-30:2.05) {};
  \node[vertex] (b4) at (-30:2.75) {};
  \draw (v4) -- (a4) -- (b4);
  \node[above left] at (v0) {\(v_0\)};
  \node[below left] at (v4) {\(v_4\)};
\end{tikzpicture}
\caption{The root graph in Example~\ref{ex:unicyclicCounterexample}.
  Two paths of length two are attached to a 12-cycle at vertices separated
  by four cycle edges.  Its line graph has an acyclic sink polynomial with
  nonreal zeros.}
\label{fig:unicyclicCounterexample}
\end{figure}
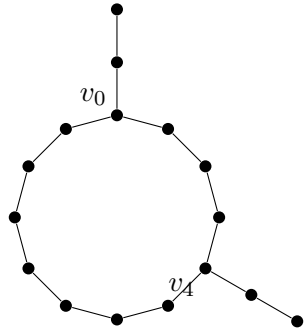

We now compute an explicit formula for the acyclic sink polynomial of a cycle, and show that this polynomial is real-rooted. For the cycle \(C_n\), where \(n\geq3\), each local order chooses the
direction of one edge of \(L(C_n)=C_n\).  Encode an orientation by a cyclic
binary word, according to whether each edge points clockwise or
counterclockwise.  Every nonconstant word has equally many transitions of
both types, so \(k\) sinks correspond to \(2k\) transitions.  We may choose
the transition positions in \(\binom{n}{2k}\) ways and the first bit in two
ways.  The two constant words have no sinks.  Consequently,
\[
N_{C_n}(t)
=(1+\sqrt{t})^n+(1-\sqrt{t})^n
=2\sum_{k\geq 0}\binom{n}{2k}t^k.
\]
In particular,
\[
N_{C_3}(t)=2+6t,
\qquad
N_{C_4}(t)=2+12t+2t^2.
\]
\begin{proposition}\label{prop:cycleSinks}
For every \(n\geq3\),
\[
S_{C_n}^{a}(t)=2\sum_{k\geq1}\binom{n}{2k}t^k
\]
is real-rooted.
\end{proposition}

Recall that a polynomial $f(t) \in \mathbb{R}[t]$ is called \defin{Hurwitz} if $f(t) = 0$ implies $\mathrm{Re}(t) \leq 0$. To prove the real-rootedness of $S_{C_n}^{a}(t)$, we need the following version of Hermite--Biehler theorem:

\begin{lemma}[Hermite--Biehler~\cite{hermite}]
\label{lem:HB-theorem}
    Let $f(t), g(t), h(t) \in \mathbb{R}[t]$ be such that $f(t) = h(t^2) + t \cdot g(t^2)$, and assume that $f(t)$ has no zeros on the imaginary axis. Then $f(t)$ is Hurwitz if and only if $g(t)$ and $h(t)$ have negative and simple roots, $g(t)$ strictly interlaces $h(t)$, and $h(0) g(0) > 0$.
\end{lemma}

\begin{proof}[Proof of Proposition \ref{prop:cycleSinks}]
    Since \(L(C_n)=C_n\), local orders on \(C_n\) correspond to all orientations of \(C_n\).  The two directed cycles are the only cyclic orientations, and therefore
    \[
        S_{L(C_n)}^{a}(t)=N_{C_n}(t)-2.
    \]
    To prove the real-rootedness, let $n \geq 2$ and consider the polynomial
    \[
        H_n(z) = \frac{(1+z)^n-1}{z}. 
    \]
    Its zeros are $z_j = e^{i{\frac{2 \pi j}{n}}} - 1$, $j = 1, \ldots, n-1$. Since $\left| z_j \right| = 1$ and $\mathrm{Re} \left(  z_j \right) < 1$, these zeros have strictly negative real part. Hence $H_n(z)$ is Hurwitz with no zeros on the imaginary axis. Moreover,
    \[
        H_n(z) = \sum_{k=1}^n \binom{n}{k} z^{k-1} = A_n(z^2) + z \cdot B_n(z^2),
    \]
    where
    \[
        A_n(z) = \sum_{k \geq 0} \binom{n}{2k+1} z^k \quad \mbox{and} \quad B_n(z) = \sum_{k \geq 1} \binom{n}{2k} z^{k-1} .
    \]
    Hence, by Lemma \ref{lem:HB-theorem}, $A_n(t)$ and $B_n(t)$ have only simple negative roots. Since $S^a_{C_n}(t) = 2t \cdot B_n(t)$, it follows that $S^a_{C_n}(t)$ is real-rooted.
\end{proof}

Theorem~\ref{thm:chordalAcyclic} gives real-rootedness of $S^a_{L(G)}(t)$ whenever the line graph is chordal. For a forest, Theorem~\ref{thm:forestEquality} gives the stronger identification with a translated weighted matching polynomial. Cycles give another real-rooted family. Example~\ref{ex:unicyclicCounterexample} shows that neither line-graph structure nor unicyclicity suffices in general. It is therefore natural to ask the following:

\begin{question}
\label{ques:positiveClass}
    For which root graphs \(G\) is \(S_{L(G)}^{a}(t)\) real-rooted?
\end{question}

\section{The sink polynomial of claw-free graphs}
\label{sec:allOrient}

In this section, we prove that the sink polynomial of claw-free graphs is real-rooted. The proof comes from an easy specialization of $\mathrm{I}_G(t; \eta)$. We also extend the notion of frozen sets to the sink polynomial, and prove the real-rootedness of this new polynomial for claw-free graphs. 

Our first result is an expression of $S_G(t)$ as a specialization of $\mathrm{I}_G(t; \eta)$. When cycles are allowed, prescribing a sink set forces its incident edge directions and leaves all other directions independent.  This gives product weights for every graph, without an elimination order. For a line graph, the weighted independence polynomial in \eqref{eq:allOrientWeightedIndependent} is a weighted matching polynomial of
the root graph.

\begin{proposition}\label{prop:allOrientIndependent}
    For every graph \(G\),
    \begin{equation}\label{eq:allOrientIndependent}
        S_G(t) = \sum_{I\in\indsets(G)} 2^{|\setE(G)|-\sum_{v\in I}d_G(v)}(t-1)^{|I|}.
    \end{equation}
    Equivalently,
    \begin{equation}
    \label{eq:allOrientWeightedIndependent}
        \frac{S_G(1+t)}{2^{|\setE(G)|}} = \operatorname{I}_G \left(t;\left\{2^{-d_G(v)}\right\}_{v\in\setV(G)}\right).
    \end{equation}
    In particular,
    \begin{equation}
    \label{eq:allOrientMatching}
        \frac{S_{L(G)}(1+t)}{2^{|\setE(L(G))|}} = \mu_G\left(t;\left\{2^{-d_{L(G)}(e)}\right\}_{e\in\setE(G)}\right).
    \end{equation}
\end{proposition}

\begin{proof}
    Fix an orientation \(\mathcal{O}\) of $G$. Expanding
    \[
        t^{\sink(\mathcal{O})}=\prod_v\bigl(1+(t-1) \cdot \mathbf{1}_{v \in \operatorname{Sink}(\mathcal{O})} \bigr)
    \]
    and summing over all orientations, the coefficient attached to \(I\subseteq\setV(H)\) counts the orientations in which every vertex of \(I\) is a sink. This count is zero unless \(I\) is independent, see proof of \eqref{eq:acyclicIndependentExpansion}.  If \(I\) is independent, its incident edges are distinct and their \(\sum_{v\in I}d_G(v)\) directions are forced.  Every remaining edge has two possible directions, which proves \eqref{eq:allOrientIndependent}.

    Now, to prove \eqref{eq:allOrientMatching}, observe that independent sets of \(L(G)\) are precisely matchings of \(G\). Substituting this correspondence into \eqref{eq:allOrientIndependent}, factoring \(2^{|\setE(L(G))|}\), and replacing \(t\) by \(1+t\) gives \eqref{eq:allOrientMatching}.
\end{proof}

\begin{example}
    As we computed in Example \ref{ex:triangle}, $S_{K_3}(t) = 2 + 6t$. Since $|E(K_3)| = 3$, $d_{K_3}(v) = 2$ for all $v \in V(K_3)$ and $I \in \mathcal{I}(K_3)$ if and only if $I = \varnothing$ or $|I| = 1$, it follows that
    \[
        2^{|E(K_3)|} \cdot \operatorname{I}_{K_3} \left(t;\left\{ 2^{-2} \right\}_{v\in\setV(K_3)}\right) = 8 \cdot \left( 1 + \frac{3t}{4} \right) = S_{K_3}(1 + t).
    \]
\end{example}

\begin{theorem}\label{thm:clawFreeAllOrient}
    If \(G\) is claw-free, then \(S_G(t)\) is real-rooted.
\end{theorem}

\begin{proof}
    Equation~\eqref{eq:allOrientWeightedIndependent} identifies a translate and positive scalar multiple of \(S_G(t)\) with a weighted independence polynomial of \(G\).  The weights are positive, so Theorem~\ref{thm:weightedClawFree} proves the claim.
\end{proof}

\begin{remark}
\label{rem:allNotAcyclic}
    Restricting to acyclic orientations can destroy real-rootedness even for a line graph, as Example~\ref{ex:unicyclicCounterexample} shows.
\end{remark}

For a frozen set \(F\), define
\[
    \defin{S_G(t;F)} \coloneqq \sum_{\substack{\mathcal{O} \in \Or (G)\\ \mathcal{O}\supseteq F}} t^{\sink(\mathcal{O})}.
\]
As before, this polynomial is zero when \(F\) is inconsistent.

Suppose that \(F\) is consistent.  Let \(\defin{\setE_F(G)} \) be the set of unfrozen edges of \(G\), that is, the set of edges of $G$ for which no directed version belongs to $F$.  Let \(\defin{W_F}\subseteq\setV(G)\) consist of the vertices with no
frozen edge directed away from them, and let \(\defin{r_F(v)}\) be the number of
unfrozen edges incident to \(v\).

\begin{proposition}
\label{prop:frozenAllOrientIndependent}
    For a graph \(G\) and a consistent frozen set \(F\),
    \begin{equation}
    \label{eq:frozenIndependent}
        S_G(t;F) = 2^{|\setE_F(G)|} \sum_{I\in\indsets(G[W_F])}(t-1)^{|I|} \prod_{v\in I}2^{-r_F(v)}.
    \end{equation}
\end{proposition}

\begin{proof}
    Use the indicator expansion from Proposition~\ref{prop:allOrientIndependent}.  Any sink set must be an independent subset of \(W_F\).  For such a set \(I\), precisely \(\sum_{v\in I}r_F(v)\) free edge directions are forced; the other free edges may be oriented independently.  This gives \eqref{eq:frozenIndependent}.
\end{proof}

\begin{example}
\label{ex:K3-frozen}
    Let $F$ be the frozen directed edge of $K_3$ colored in red in Figure \ref{fig:K3-frozen}, which shows all orientations of $K_3$ containing $F$. Clearly, $S_{K_3} (t; F) = 1 + 3t$. On the other hand, $|E_F(K_3)| = 2$, $W_F = \{b,c\}$, $r_F(b) = 1$ and $r_F(c) = 2$. Since $K_3[W_F]$ is the graph on $\{b,c\}$ whose only edge set is $\{bc\}$, it follows that its independent sets are $\varnothing$, $\{b\}$ and $\{c\}$. Hence,
    \[
        2^{|\setE_F(K_3)|} \sum_{I\in\indsets(K_3[W_F])}(t-1)^{|I|} \prod_{v\in I}2^{-r_F(v)} = 4 \left( 1 + \frac{t-1}{2} + \frac{t-1}{4}  \right) = S_{K_3} (t; F).
    \]
    
    \begin{figure}[H]
        \centering
        \begin{tikzpicture}[scale=0.85, vertex/.style={circle,fill=black,inner sep=2.2pt},edge label/.style={fill=white,inner sep=1.5pt,font=\small},>=Stealth]
            \node[scale=0.5, vertex, label=below:\(a\)] (v1) at (0,0) {};
            \node[scale=0.5, vertex, label=below:\(c\)] (v2) at (2,0) {};
            \node[scale=0.5, vertex, label=above:\(b\)] (v3) at (1,1.732) {};
            \draw[postaction={decorate},decoration={markings,mark=at position 0.6 with {\arrow{>}}}] (v1) -- (v2);
            \draw[postaction={decorate},decoration={markings,mark=at position 0.6 with {\arrow{>}}}] (v3) -- (v2);
            \draw[postaction={decorate},decoration={markings,mark=at position 0.6 with {\arrow{>}}},red, dashed] (v1) -- (v3);
    
            \node[scale=0.5, vertex, label=below:\(a\)] (v4) at (4,0) {};
            \node[scale=0.5, vertex, label=below:\(c\)] (v5) at (6,0) {};
            \node[scale=0.5, vertex, label=above:\(b\)] (v6) at (5,1.732) {};
            \draw[postaction={decorate},decoration={markings,mark=at position 0.6 with {\arrow{>}}}] (v4) -- (v5);
            \draw[postaction={decorate},decoration={markings,mark=at position 0.6 with {\arrow{>}}}] (v5) -- (v6);
            \draw[postaction={decorate},decoration={markings,mark=at position 0.6 with {\arrow{>}}},red, dashed] (v4) -- (v6);
    
            \node[scale=0.5, vertex, label=below:\(a\)] (v7) at (8,0) {};
            \node[scale=0.5, vertex, label=below:\(c\)] (v8) at (10,0) {};
            \node[scale=0.5, vertex, label=above:\(b\)] (v9) at (9,1.732) {};
            \draw[postaction={decorate},decoration={markings,mark=at position 0.6 with {\arrow{>}}}] (v8) -- (v7);
            \draw[postaction={decorate},decoration={markings,mark=at position 0.6 with {\arrow{>}}}] (v9) -- (v8);
            \draw[postaction={decorate},decoration={markings,mark=at position 0.6 with {\arrow{>}}},red, dashed] (v7) -- (v9);
    
            \node[scale=0.5, vertex, label=below:\(a\)] (v19) at (12,0) {};
            \node[scale=0.5, vertex, label=below:\(c\)] (v20) at (14,0) {};
            \node[scale=0.5, vertex, label=above:\(b\)] (v21) at (13,1.732) {};
            \draw[postaction={decorate},decoration={markings,mark=at position 0.6 with {\arrow{>}}}] (v20) -- (v19);
            \draw[postaction={decorate},decoration={markings,mark=at position 0.6 with {\arrow{>}}}] (v20) -- (v21);
            \draw[postaction={decorate},decoration={markings,mark=at position 0.6 with {\arrow{>}}},red, dashed] (v19) -- (v21);
        \end{tikzpicture}
        \caption{All the orientations of the triangle $K_3$ containing $F$, the frozen set defined in Example \ref{ex:K3-frozen}.}
        \label{fig:K3-frozen}
    \end{figure}
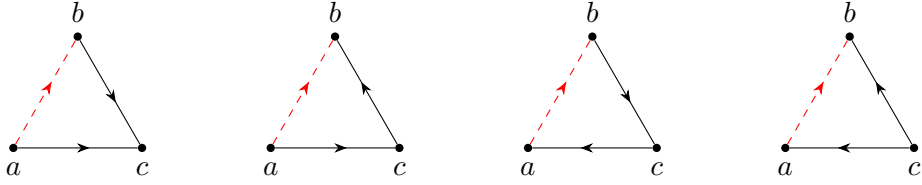
\end{example}

\begin{corollary}
\label{cor:frozenClawFree}
    Let \(G\) be claw-free and let \(F\) be a frozen set in \(G\). Then \(S_G(t;F)\) is either zero or real-rooted.
\end{corollary}

\begin{proof}
    The polynomial is zero if \(F\) is inconsistent.  Otherwise, \eqref{eq:frozenIndependent}, after replacing \(t\) by \(1\), is a positive scalar multiple of the weighted independence polynomial of \(G[W_F]\), with vertex weights \(2^{-r_F(v)}\).  Induced subgraphs of claw-free graphs are claw-free, so Theorem~\ref{thm:weightedClawFree} proves the claim.
\end{proof}

For a graph \(G\) and a frozen set \(F\) in \(L(G)\), let
\(\defin{\mathcal E(F)}\subseteq\setE(G)\) be the root edges whose
line-graph vertices lie in \(W_F\), and let \(\defin{G(F)}\) be the spanning
subgraph of \(G\) with edge set \(\mathcal E(F)\).  These are the eligible
sink edges of the root graph.

\begin{corollary}
\label{cor:frozenAllOrientMatching}
    Let \(G\) be a graph and let \(F\) be a frozen set in \(L(G)\).  Then \(S_{L(G)}(t;F)\) is zero if \(F\) is inconsistent.  If \(F\) is
    consistent, then
    \begin{equation}
    \label{eq:frozenAllOrientMatching}
        \frac{S_{L(G)}(1+t;F)}{2^{|\setE_F(L(G))|}} = \mu_{G(F)}\left(t;\left\{2^{-r_F(e)}\right\}_{e\in\mathcal E(F)}\right).
    \end{equation}
    In particular, every nonzero \(S_{L(G)}(t;F)\) is real-rooted.
\end{corollary}

\begin{proof}
    There is nothing to prove when \(F\) is inconsistent.  Suppose that \(F\) is
    consistent.  Then
    \[
        L(G)[W_F]=L(G(F)).
    \]
    Hence \eqref{eq:frozenIndependent}, after replacing \(t\) by \(1+t\), becomes \eqref{eq:frozenAllOrientMatching}. Theorem~\ref{thm:heilmannLieb} proves the claim.
\end{proof}

\begin{remark}
    Since $L(G)$ is claw-free for every graph $G$, the real-rootedness of $S_{L(G)}(t; F)$ follows directly from Corollary \ref{cor:frozenClawFree}. However, Corollary \ref{cor:frozenAllOrientMatching} presents a formula for $S_{L(G)}(t; F)$ that depends only on the root graph $G$ up to a constant. 
\end{remark}

\section{Realizations and symmetric-function applications}\label{sec:applications}

In this section, we give explicit polynomial families realized by the preceding models, then
return to the chromatic and LLT correspondences motivating our results and
their interpretation in Hessenberg cohomology.  The Carlsson--Mellit relation expresses the same sink polynomial
as a plethystic specialization of a unicellular LLT polynomial.

\subsection{Polynomial families realized by graphs}
By Proposition \ref{prop:cycleSinks}, the even binomial row gives a cycle sink polynomial after its constant term is removed and the result is doubled. We present now other families of graphs which give rise to nice closed formulas for one or more of the polynomials studied in Sections \ref{sec:acyclic-pol} and \ref{sec:allOrient}.

\begin{proposition}\label{prop:explicitRealizations}
Let \(P_n\) denote the path on \(n\geq1\) vertices, let \(D_{p,q}\) be the
tree formed by joining two centers and attaching \(p,q\geq0\) leaves to them,
and let \(F_r\) be the graph formed by \(r\geq1\) triangles sharing one vertex.
Then
\begin{align*}
S_{P_n}^{a}(t)=S_{P_n}(t)=M_{P_{n+1}}(t)
  &=\sum_{k\geq1}\binom{n}{2k-1}t^k,\\
S_{K_n}^{a}(t)=M_{K_{1,n}}(t)&=n!t,\\
S_{L(D_{p,q})}^{a}(t)=M_{D_{p,q}}(t)
  &=p!q!\bigl((p+q+1)t+pq\,t^2\bigr),\\
S_{L(F_r)}^{a}(t)&=3(2r)!\,t(t+2)^{r-1}.
\end{align*}
All these polynomials are real-rooted.
\end{proposition}

\begin{proof}
For the path, encode each edge direction by a bit, with \(1\) indicating a
rightward edge.  Append a \(1\) before the first bit and a \(0\) after the
last.  A sink is a transition from \(1\) to \(0\), so \(k\) sinks require
\(2k-1\) transitions among \(n\) positions.  Their positions determine the
word.  Every orientation of a path is acyclic, and \(L(P_{n+1})=P_n\).
The complete graph has \(n!\) acyclic orientations, all with one sink; the
star identity follows from the forest correspondence.

The line graph of \(D_{p,q}\) consists of cliques of orders \(p+1\) and
\(q+1\) sharing one vertex.  There are \((p+1)!(q+1)!\) pairs of local
edge orders.  Two sinks occur exactly when the shared vertex is not minimal in
either clique, giving \(pq\,p!q!\) pairs.  Every other pair gives one sink.

For \(L(F_r)\), the \(2r\) edges incident to the common root vertex form a
clique.  Each of the remaining \(r\) vertices has two adjacent neighbors in
this clique, and these neighbor pairs are disjoint.  Insert the central
clique first and the remaining vertices last in a reverse PEO.  The
normalization in \eqref{eq:chordalSinkIdentity} is \((2r)!3^r\).
Each added vertex has weight \(1/3\), and each central vertex has weight
\(1/(3r)\).  Consequently,
\[
    \operatorname{I}_{L(F_r)}(t;\eta)
    =\left( 1+\frac{t}{3} \right)^r+\frac{2t}{3}\left( 1+\frac{t}{3} \right)^{r-1}
    =\left( 1+t \right) \left( 1+\frac{t}{3} \right)^{r-1}.
\]
Substituting $t$ by $t-1$ proves the last formula.  Real-rootedness follows
from the forest and chordal claw-free theorems.
\end{proof}

For cycles, Proposition~\ref{prop:cycleSinks} gives the even-binomial family
\(2\sum_{k\geq1}\binom{n}{2k}t^k\).  The path and cycle rows correspond to
\href{https://oeis.org/A034867}{A034867} and
\href{https://oeis.org/A034839}{A034839}, respectively.
Every path power \(P_n^r\) is also a unit interval graph, so its acyclic sink
polynomial is real-rooted by Theorem~\ref{thm:chordalAcyclic}.

\subsection{An ascent refinement}

A \defin{Hessenberg function} is a nondecreasing map \(h:[n]\to[n]\) with
\(h(i)\geq i\).  Its \defin{natural unit interval graph} \(H_h\) has edges
\(\{i,j\}\) whenever \(i<j\leq h(i)\).  Both the earlier and the later
neighbors of a vertex form cliques.  In particular, the natural order is a
reverse PEO and the graph is claw-free.

For an orientation \(\mathcal O\) of this labeled graph, let
\(\defin{\operatorname{asc}(\mathcal O)}\) count its directed edges
\(i\to j\) with \(i<j\), and define
\[
\defin{S_{H_h}^{a}(t;q)}
=\sum_{\mathcal O\in\AO(H_h)}q^{\operatorname{asc}(\mathcal O)}
 t^{\sink(\mathcal O)}.
\]
At \(q=1\), this is the ordinary acyclic sink polynomial.

\begin{theorem}\label{thm:unitIntervalAscent}
For every Hessenberg function \(h\) and every real \(q\geq0\),
\(S_{H_h}^{a}(t;q)\) has only real nonpositive zeros.
\end{theorem}

\begin{proof}
We use the marked-sink induction of Proposition~\ref{prop:markedChordalModel},
now assigning weight \(q^{\operatorname{asc}(\mathcal O)}\) to each orientation.
Let $G = H_h$ and write \(S^a_{G,W}(t;q)\) for the resulting marked polynomial.  When the last
vertex \(v\) is inserted, use the same notation \(H,K,d,T\) as in that
proof, and put \([d]_q=1+q+\dotsb+q^{d-1}\), with \([0]_q=0\).
A cut at position \(r\) adds exactly \(r\) ascents, because every old
vertex has a smaller label than \(v\).  Thus
\[
    S^a_{G,W}(t;q) = [d]_q \cdot S^a_{H,T}(t;q) + q^d (1+ t \cdot \mathbf1_{v\in W}) \cdot S^a_{H,T\setminus K}(t;q).
\]
Suppose the old marked polynomials have the simultaneous representation
\(D \cdot \operatorname{I}_{G,T}(t;\theta)\).  Set
\[
    c=[d+1]_q=[d]_q+q^d,\qquad D'=cD,
\]
\[
\theta'_v=\frac{q^d}{c},
\qquad
\theta'_u=
\begin{cases}
[d]_q\theta_u/c,&u\in K,\\
\theta_u,&u\notin K.
\end{cases}
\]
Since an independent set meets \(K\) at most once, these weights reproduce
the displayed recurrence, exactly as in
\eqref{eq:markedCliqueScaling}--\eqref{eq:markedIndependenceRecurrence}.
Starting with \(D=1\) on the empty graph proves the representation for every
marked set.  For \(q\geq0\), we have \(c>0\), \(D'>0\), and all weights
are nonnegative, including at \(q=0\).  Taking all vertices marked and
applying Theorem~\ref{thm:weightedClawFree} proves real-rootedness after
translation.  The original coefficients are nonnegative and the polynomial
is nonzero: the orientation directing every edge toward the smaller label
has ascent weight one.  Therefore no zero is positive.
\end{proof}

\subsection{Hessenberg cohomology}
Let \(X= \{x_1, x_2, \ldots \}\) be a symmetric-function alphabet, and write
\(e_\lambda\) and \(h_\lambda\) for the elementary and complete homogeneous
bases.  The \defin{chromatic quasisymmetric function} is
\[
X_{G}(X;q)=\sum_{\kappa\text{ proper coloring}}
 q^{|\{\{i,j\}\in\setE(G):i<j,\ \kappa(i)<\kappa(j)\}|}
 \prod_{i=1}^n x_{\kappa(i)}.
\]
It is symmetric for unit interval graphs, that is, for $G = H_h$ for some Hessenberg function $h$. If
\(X_{H_h}(X;q)=\sum_{\lambda\vdash n}c_\lambda(q)e_\lambda(X)\), then
J.~Shareshian and M.~L.~Wachs~\cite[Theorem~5.3]{ShareshianWachs2016} prove
\begin{equation}\label{eq:elementaryLengthSinks}
\sum_{\lambda\vdash n}c_\lambda(q)t^{\ell(\lambda)}
=S_{H_h}^{a}(t;q),
\end{equation}
where \(\ell(\lambda)\) is the number of parts of \(\lambda\).
Theorem~\ref{thm:unitIntervalAscent} therefore gives real-rootedness in the
partition-length variable for every fixed \(q\geq0\).

For a regular semisimple linear map \(A:\mathbb C^n\to\mathbb C^n\), let
\[
\operatorname{Hess}(A,h)
=\{V_\bullet:\dim V_i=i,\ A V_i\subseteq V_{h(i)}\}
\]
be the associated Hessenberg variety of complete flags.  P.~Brosnan and
T.~Y.~Chow~\cite{BrosnanChow2018}, and independently
M.~Guay-Paquet~\cite{GuayPaquet2016}, identify the graded Frobenius
characteristic of Tymoczko's dot action as
\begin{equation}\label{eq:hessenbergFrobenius}
\mathcal F_h(X;q)
\coloneqq\sum_{j\geq0}q^j
 \operatorname{ch}H^{2j}(\operatorname{Hess}(A,h);\mathbb C)
=\omega X_{H_h}(X;q),
\end{equation}
where \(\omega(e_\lambda)=h_\lambda\).

\begin{corollary}\label{cor:hessenbergLength}
Write \(\mathcal F_h(X;q)=\sum_{\lambda\vdash n}b_\lambda(q)h_\lambda(X)\).
For every real \(q\geq0\), the polynomial
\[
\sum_{\lambda\vdash n}b_\lambda(q)t^{\ell(\lambda)}
\]
is real-rooted, with nonnegative coefficients, and its coefficient sequence
is log-concave with no internal zeros.
\end{corollary}

\begin{proof}
Equations \eqref{eq:elementaryLengthSinks} and
\eqref{eq:hessenbergFrobenius} identify this polynomial with
\(S_{H_h}^{a}(t;q)\), from which the real-rootedness and log-concavity follow.
\end{proof}

Here \(q\) records cohomological half-degree and \(t\) records the number
of parts in the homogeneous-basis expansion.  Thus the conclusion concerns
this specialization of the Frobenius characteristic; the ordinary Poincar\'e
polynomial uses dimensions of the cohomology groups instead.

\subsection{The LLT connection}

In the same ascent convention, the unicellular LLT polynomial is
\[
\mathcal L_h(X;q)=\sum_{\kappa:[n]\to\mathbb Z_{>0}}
 q^{|\{\{i,j\}\in\setE(H_h):i<j,\ \kappa(i)<\kappa(j)\}|}
 \prod_{i=1}^n x_{\kappa(i)}.
\]
The sum includes all colorings.  Let \(\Phi_t\) be the homomorphism defined
by \(\Phi_t(e_m)=t\) for \(m\geq1\).  The Carlsson--Mellit
relation~\cite[Lemma~6.1]{AlexanderssonPanova2018} gives
\[
\mathcal L_h[(q-1)X;q]=(q-1)^nX_{H_h}(X;q).
\]
Here plethystic substitution means \(p_k[(q-1)X]=(q^k-1)p_k[X]\).
Consequently,
\[
S_{H_h}^{a}(t;q)
=\Phi_t\!\left((q-1)^{-n}\mathcal L_h[(q-1)X;q]\right).
\]
Theorem~\ref{thm:unitIntervalAscent} proves that this specialization is
real-rooted in \(t\) for every fixed \(q\geq0\).
To recover the ordinary sink polynomial
at \(q=1\), we first cancel \((q-1)^n\) using this polynomial identity.

\subsubsection*{Half-sink products}

Half-sinks give a direct product in both orientation classes.  Expand the
shifted LLT polynomial as
\[
\mathcal L_h(X;1+u)
=\sum_{\lambda\vdash n}d_\lambda(u)e_\lambda(X).
\]
A \defin{half-sink} of an orientation is a vertex with no outgoing edge to a
larger-labeled vertex.  Write \(\operatorname{hsink}(\mathcal O)\) for their
number.  P.~Alexandersson and G.~Panova~\cite[Proposition~5.2]{AlexanderssonPanova2018}
prove
\begin{equation}\label{eq:lltHalfSinks}
\sum_{\lambda\vdash n}d_\lambda(u)t^{\ell(\lambda)}
=\sum_{\mathcal O \in \Or(H_h)}
 u^{\operatorname{asc}(\mathcal O)}t^{\operatorname{hsink}(\mathcal O)}.
\end{equation}

\begin{proposition}\label{prop:lltLengthProduct}
Let \(a_i=h(i)-i\) be the number of larger-labeled neighbors of \(i\).
For \(z\geq0\), we have
\begin{align*}
\sum_{\mathcal O \in \Or(H_h)}
 z^{\operatorname{asc}(\mathcal O)}t^{\operatorname{hsink}(\mathcal O)}
 &=\prod_{i=1}^n\bigl(t+(1+z)^{a_i}-1\bigr),\\
\sum_{\mathcal O\in\AO(H_h)}
 z^{\operatorname{asc}(\mathcal O)}t^{\operatorname{hsink}(\mathcal O)}
 &=\prod_{i=1}^n\bigl(t+z+z^2+\dotsb+z^{a_i}\bigr),
\end{align*}
where an empty sum is zero.  In particular,
\begin{equation}\label{eq:lltLengthProduct}
\sum_{\lambda\vdash n}d_\lambda(u)t^{\ell(\lambda)}
=\prod_{i=1}^n\bigl(t+(1+u)^{a_i}-1\bigr).
\end{equation}
For every real \(u\geq0\), its zeros are real and nonpositive.
\end{proposition}

\begin{proof}
For the first product, group the edges by their smaller endpoint.  The
\(a_i\) edges in group \(i\) can be oriented independently of every other group.  Choosing no
ascending edge contributes \(t\), since \(i\) is then a half-sink.  The
nonempty choices contribute
\(\sum_{k=1}^{a_i}\binom{a_i}{k}z^k=(1+z)^{a_i}-1\).
Multiplying these contributions and using \eqref{eq:lltHalfSinks} proves the
identity.  The roots are \(1-(1+u)^{a_i}\leq0\).

For the acyclic product, insert vertices in decreasing label order.  The
larger-labeled neighbors of \(i\) form a clique, with a total order in every
acyclic orientation.  We may insert \(i\) at any of its \(a_i+1\) cuts.
The cut with no outgoing edge makes \(i\) a half-sink and contributes
\(t\); the other cuts contribute \(z,z^2,\dotsc,z^{a_i}\).
Each old vertex gains only a neighbor with a smaller label, so its half-sink
status is unchanged.  Multiplying the insertion factors proves the formula.
\end{proof}

At \(z=1\), the two half-sink polynomials are
\[
\prod_i(t+2^{a_i}-1)
\qquad\text{and}\qquad
\prod_i(t+a_i),
\]
respectively.  Their factors are read directly from the area sequence of the unit interval graph \cite{AlexanderssonPanova2018}.

\section{X-descents, sink extraction, and real-rooted specializations}
\label{sec:xDescents}

In this section, we show that we can recover both sink polynomials and \(X\)-Eulerian polynomials from
fundamental quasisymmetric expansions.  The distinction lies in the linear
map we apply: fundamental length records the number of descents, whereas
Stanley's map extracts specified descent sets.  On symmetric functions,
the latter agrees with the elementary-length homomorphism \(\Phi_t\).

\subsection{Fundamental length and principal specialization}

For a composition \(\alpha\) of \(n\geq1\), let
\(S(\alpha)\subseteq[n-1]\) be its set of proper partial sums.  We use
Gessel's fundamental quasisymmetric function
\[
\defin{F_\alpha}=F_{n,S(\alpha)}
=\sum_{\substack{1 \leq i_1\leq\dotsb\leq i_n\\
                 j\in S(\alpha)\Rightarrow i_j<i_{j+1}}}
 x_{i_1}\dotsm x_{i_n}.
\]
We define the linear map \(\Psi_t\) on positive-degree quasisymmetric
functions by
\[
\defin{\Psi_t(F_\alpha)}\coloneqq t^{\ell(\alpha)}
=t^{|S(\alpha)|+1}.
\]
We set \(\Psi_t(1)=1\).  This map is not multiplicative: the identity
\(F_{1,\varnothing}^2=F_{2,\varnothing}+F_{2,\{1\}}\) gives
\[
\Psi_t(F_{1,\varnothing}^2)=t+t^2,
\qquad \Psi_t(F_{1,\varnothing})^2=t^2.
\]
In contrast, \(\Phi_t\) is the algebra homomorphism on symmetric functions
determined by \(\Phi_t(e_m)=t\), so
\(\Phi_t(e_\lambda)=t^{\ell(\lambda)}\).

\begin{proposition}\label{prop:fundamentalChromatic}
    If \(f\) is homogeneous of degree \(n\geq1\), then
    \begin{equation}\label{eq:fundamentalPrincipal}
        \Psi_t(f)=(1-t)^{n+1}\sum_{m\geq0}f(1^m)t^m,
    \end{equation}
    where $1^m = (\overbrace{1, \ldots, 1}^{m}, 0, 0, \ldots)$. Consequently, for a graph \(G\) on \(n\) vertices,
    \begin{equation}
    \label{eq:chromaticWPolynomial}
        \Psi_t(X_G)=(1-t)^{n+1}\sum_{m\geq0}\chi_G(m)t^m=w_G(t),
    \end{equation}
    where \(w_G\) is the \(w\)-polynomial of F.~Brenti~\cite[Proposition~4.1]{Brenti1992}. For homogeneous symmetric \(f\), we also have
    \begin{equation}
    \label{eq:fundamentalOmega}
        \Psi_t(\omega f)=t^{n+1}\Psi_{1/t}(f).
    \end{equation}
\end{proposition}

\begin{proof}
If \(|S|=d\), then
\(F_{n,S}(1^m)=\binom{m+n-1-d}{n}\), so
\[
\sum_{m\geq0}F_{n,S}(1^m)t^m
=\frac{t^{d+1}}{(1-t)^{n+1}}.
\]
Linearity proves \eqref{eq:fundamentalPrincipal}, and
\(X_G(1^m)=\chi_G(m)\) gives \eqref{eq:chromaticWPolynomial}.
For \eqref{eq:fundamentalOmega}, expand a Schur function in the fundamental
basis over standard Young tableaux.  Transposing a tableau complements its
descent set in \([n-1]\), changing the exponent \(d+1\) to \(n-d\).
Since \(\omega s_\lambda=s_{\lambda'}\), linearity gives the result.
\end{proof}

\begin{corollary}[\cite{Brenti1992}]\label{cor:fundamentalChordal}
For every nonempty chordal graph \(G\), both \(\Psi_t(X_G)\) and
\(\Psi_t(\omega X_G)\) have only real nonpositive zeros.
In particular, this holds for every natural unit interval graph.
\end{corollary}

\begin{proof}
Brenti proves real-rootedness of \(w_G\) for supersolvable graphs
\cite[Theorem~4.11]{Brenti1992}, which are exactly the chordal graphs
\cite[Theorem~4.15]{Brenti1992}.  Its coefficients are nonnegative by
\cite[Proposition~4.5]{Brenti1992}.  Apply
\eqref{eq:chromaticWPolynomial} and \eqref{eq:fundamentalOmega}.
\end{proof}

The involution \(\omega\) therefore reverses the fundamental-length
polynomial, with exponent \(n+1\) in our normalization.  If we instead
send \(F_{n,S}\) to \(t^{|S|}\), the reversal exponent is \(n-1\).
This complementation accounts for the passage from \(X_G\) to
\(\omega X_G\); the map that recovers sinks uses different weights.

\subsection{Stanley's sink map}

In the proof of his sink theorem, R.~P.~Stanley
\cite[Theorem~3.3, p.~177]{Stanley1995} defines a linear map on fundamental
quasisymmetric functions.  In our notation, we write
\begin{equation}\label{eq:stanleySinkMap}
\defin{\varphi_t(F_{n,S})}\coloneqq
\begin{cases}
t(t-1)^i,& S=\{i+1,\dotsc,n-1\},\quad 0\leq i\leq n-1,\\
0,&\text{otherwise},
\end{cases}
\end{equation}
and set \(\varphi_t(1)=1\).  Thus the full descent set contributes \(t\),
and the empty descent set contributes \(t(t-1)^{n-1}\).

\begin{proposition}[\cite{Stanley1995}]\label{prop:stanleySinkMap}
The restriction of \(\varphi_t\) to symmetric functions is \(\Phi_t\).
For every nonempty graph \(G\), we have
\begin{equation}\label{eq:stanleySinkExtraction}
\varphi_t(X_G)=\Phi_t(X_G)
=\sum_{\mathcal O\in\AO(G)}t^{\sink(\mathcal O)}
=S_G^{a}(t).
\end{equation}
\end{proposition}

\begin{proof}
We follow Stanley's proof through strict order-preserving maps of posets.
Let \(K_P\) enumerate these maps for a nonempty poset \(P\), and suppose
\(P\) has \(m\) minimal elements.  Choose an order-reversing labeling of
\(P\) by \([n]\).  The fundamental expansion of \(K_P\) is the sum of
\(F_{n,\operatorname{Des}(\sigma)}\) over the labeled linear extensions.

An extension with descent set \(\{i+1,\dotsc,n-1\}\) consists of an
increasing prefix, the largest label \(n\), and a decreasing suffix.
Every element before \(n\) must be minimal: a predecessor would have a
larger label and would have to appear earlier in the increasing prefix.
Conversely, choose any \(i\) of the other \(m-1\) minimal elements, list
them in increasing label order, then list \(n\) and all remaining elements
in decreasing label order.  This gives all such extensions.  Hence
\[
\varphi_t(K_P)
=\sum_{i=0}^{m-1}\binom{m-1}{i}t(t-1)^i=t^m.
\]
For a disjoint union of chains of sizes \(\lambda_1,\dotsc,\lambda_k\),
we have \(K_P=e_\lambda\) and \(m=k\).  This proves
\(\varphi_t(e_\lambda)=t^{\ell(\lambda)}\), and therefore
\(\varphi_t|_{\mathrm{Sym}}=\Phi_t\).

Orient each edge of a properly colored graph from its larger color to its
smaller color.  These orientations are acyclic.  For a fixed orientation,
order the vertices so that a directed edge points from a larger to a smaller
element.  Its compatible colorings are precisely the strict order-preserving
maps, and its minimal elements are the sinks.  Applying the preceding
identity to each orientation proves \eqref{eq:stanleySinkExtraction}.
\end{proof}

\begin{remark}
    The maps \(\Psi_t\) and \(\Phi_t\) have different algebraic roles: \(\Psi_t\) records fundamental length linearly, \(\Phi_t\) records elementary length multiplicatively, and \(\varphi_t\) extends \(\Phi_t\) linearly to quasisymmetric functions.
\end{remark}

\subsection{X-descent sets and sinks}

We now apply both maps to the same descent expansion.  For a relation
\(R\subseteq[n]\times[n]\), put
\[
\defin{\operatorname{Des}_R(\sigma)}
\coloneqq\{j\in[n-1]:(\sigma_j,\sigma_{j+1})\in R\},
\qquad
\defin{A_R(t)}\coloneqq
\sum_{\sigma\in\mathfrak S_n}t^{|\operatorname{Des}_R(\sigma)|}.
\]
This is the \(X\)-descent statistic, with the relation denoted by \(R\)
to distinguish it from the symmetric-function alphabet.  In the
R\'edei--Berge function of D.~Grinberg and
R.~P.~Stanley~\cite[Definitions~1.6 and~1.16]{GrinbergStanley2023},
\[
\defin{U_R}\coloneqq
\sum_{\sigma\in\mathfrak S_n}F_{n,\operatorname{Des}_R(\sigma)},
\qquad \Psi_t(U_R)=tA_R(t).
\]

Let \(P\) be a poset on \([n]\), let \(H\) be its incomparability graph,
and take \(R=\{(a,b):a>_P b\}\).  Shareshian and Wachs use decreasing
indices in their fundamental functions
\cite[Section~3, footnote~3]{ShareshianWachs2016}.  Reversing compositions
converts their convention to ours and fixes symmetric functions.
Their fundamental expansion~\cite[Theorem~3.1]{ShareshianWachs2016},
at unit ascent weight, therefore identifies
\(U_R=\omega X_H\).  On symmetric functions, \(\omega\) is induced by
the linear operation \(F_{n,S}\mapsto F_{n,[n-1]\setminus S}\), as follows
from the tableau argument in Proposition~\ref{prop:fundamentalChromatic}.
We obtain
\[
X_H=\sum_{\sigma\in\mathfrak S_n}
F_{n,[n-1]\setminus\operatorname{Des}_R(\sigma)}.
\]
An initial interval \([i]=\{1,\dotsc,i\}\), with \([0]=\varnothing\),
has complement \(\{i+1,\dotsc,n-1\}\).  Applying
Proposition~\ref{prop:stanleySinkMap} gives the following formula:

\begin{corollary}\label{cor:xDescentSinks}
For the incomparability graph \(H\) of a poset \(P\) on \([n]\), with
\(n\geq1\) and \(R=\{(a,b):a>_P b\}\), we have
\begin{equation}\label{eq:xDescentSinks}
S_H^{a}(t)
=t\sum_{i=0}^{n-1}
\bigl|\{\sigma\in\mathfrak S_n:
\operatorname{Des}_R(\sigma)=[i]\}\bigr|(t-1)^i.
\end{equation}
\end{corollary}

\begin{example}\label{ex:xDescentPath}
For the path \(P_3\) with edges \(\{1,2\}\) and \(\{2,3\}\), we
take the poset whose only strict relation is \(1<_P 3\), so
\(R=\{(3,1)\}\).  Four permutations have empty \(R\)-descent set;
\(\mathtt{312}\) has descent set \(\{1\}\), and \(\mathtt{231}\) has descent set
\(\{2\}\).  Hence
\[
\omega X_{P_3}=4F_{3,\varnothing}+F_{3,\{1\}}+F_{3,\{2\}},
\qquad \Psi_t(\omega X_{P_3})=4t+2t^2.
\]
For sink extraction, the initial-interval coefficients are \(4,1,0\).
Thus
\[
\varphi_t(X_{P_3})=4t+t(t-1)=3t+t^2
=\Phi_t(X_{P_3}).
\]
Three orientations of this path have one sink, and one has two sinks.
Fundamental length combines the two singleton descent sets, whereas
Stanley's extraction assigns them different contributions.
\end{example}

For \(H_h\), let \(P_h\) be the natural unit interval order defined by
\(i<_{P_h}j\) if \(h(i)<j\), and take
\(R_h=\{(i,j):i>_{P_h}j\}=\{(i,j):i>h(j)\}\).
The fundamental expansion of Shareshian and Wachs
\cite[Theorem~3.1]{ShareshianWachs2016} gives
\[
\omega X_{H_h}(X;q)
=\sum_{\sigma\in\mathfrak S_n}
 q^{\operatorname{inv}_{H_h}(\sigma)}F_{n,\operatorname{Des}_{R_h}(\sigma)},
\]
where \(\operatorname{inv}_{H_h}(\sigma)\) counts pairs of positions
\(i<j\) for which \(\sigma_i>\sigma_j\) and
\(\{\sigma_i,\sigma_j\}\) is an edge.  The symmetry of
\(X_{H_h}(X;q)\), proved in \cite[Theorem~4.5]{ShareshianWachs2016},
justifies the same conversion of fundamental conventions for every \(q\).
Hence
\begin{equation}\label{eq:xDescentSpecialization}
\Psi_t(\omega X_{H_h}(X;q))
=t\sum_{\sigma\in\mathfrak S_n}
 q^{\operatorname{inv}_{H_h}(\sigma)}
 t^{|\operatorname{Des}_{R_h}(\sigma)|}.
\end{equation}
At \(q=1\), Corollary~\ref{cor:fundamentalChordal} proves real-rootedness
of the \(X\)-Eulerian polynomial \(A_{R_h}(t)\).
Applying \(\varphi_t\omega\) instead gives
\begin{equation}\label{eq:xDescentWeightedSinks}
S_{H_h}^{a}(t;q)
=t\sum_{i=0}^{n-1}(t-1)^i
\sum_{\substack{\sigma\in\mathfrak S_n\\
\operatorname{Des}_{R_h}(\sigma)=[i]}}
q^{\operatorname{inv}_{H_h}(\sigma)}.
\end{equation}
Here we use \(\varphi_t|_{\mathrm{Sym}}=\Phi_t\) and the ascent-refined
sink theorem \cite[Theorem~5.3]{ShareshianWachs2016}.  Theorem~\ref{thm:unitIntervalAscent}
proves real-rootedness of this sink specialization for every fixed \(q\geq0\).

For comparison, \(C_4\) has
\[
\Psi_t(X_{C_4})=2t^2+8t^3+14t^4,
\qquad
\Psi_t(\omega X_{C_4})=14t+8t^2+2t^3.
\]
Their quadratic factors have discriminant \(-48\).  Its ordinary acyclic
sink polynomial is \(12t+2t^2\), illustrating the distinction between
the real-rootedness behavior of the two maps.

\subsection{Further questions}

We write \(\mathrm{QSym}_n\) for the real vector space of homogeneous
quasisymmetric functions of degree \(n\), and ask which weights on the
fundamental basis give real-rooted images of a specified family.

\begin{question}\label{ques:realRootedSpecializations}
For a specified family \(\mathcal F_n\subseteq\mathrm{QSym}_n\), classify
the linear maps
\[
\defin{L_n(F_{n,S})}\coloneqq w_{n,S}(t)\in\mathbb R[t],
\qquad \deg w_{n,S}\leq n,
\]
normalized by \(L_1(F_{1,\varnothing})=t\), such that every
\(L_n(f)\), for \(f\in\mathcal F_n\), is either zero or has
nonnegative coefficients and only real zeros.  We ask this separately for
the family of chromatic functions of chordal graphs, its \(\omega\)-image,
and the family \(\{U_R:R\subseteq[n]\times[n]\}\).
For the last family, we also ask the question with \(R\) restricted to
strict partial orders.
For the family \(X_{H_h}(X;q)\) of natural unit interval graphs, which maps
have this property for every fixed real \(q\geq0\)?  In particular, is
\(\Psi_t(X_{H_h}(X;q))\) real-rooted for every such \(h\) and \(q\)?
\end{question}

Complementation gives one constraint to study.  If the weights satisfy
\[
w_{n,[n-1]\setminus S}(t)=t^{d_n}w_{n,S}(1/t)
\]
for a fixed integer \(d_n\), then the corresponding images of \(f\) and
\(\omega f\), for symmetric \(f\), are reciprocal polynomials.
For \(\Psi_t\) we have \(d_n=n+1\), so the two versions of the final
question are equivalent.

Multiplicativity supplies a different mechanism: an algebra homomorphism
on symmetric functions sends the chromatic function of a disjoint union to
a product.  Stanley's map realizes such a homomorphism by linear extraction
on the larger space of quasisymmetric functions.  We can therefore ask for
conditions on \(w_{n,S}\) that make the restrictions of the \(L_n\)
multiplicative, or that yield real-rootedness directly from descent-set
coefficients.  Principal-specialization formulas such as
\eqref{eq:fundamentalPrincipal} give a third route.  When a family admits a
stable multivariate refinement, a specialization preserving stability could
provide sufficient conditions for its univariate images to be real-rooted.

At \(q=1\), Brenti's theorem applies to \(\Psi_t\) on chordal chromatic
functions.  For \(\Phi_t\), Theorem~\ref{thm:chordalAcyclic} treats
chordal claw-free graphs, and Theorem~\ref{thm:unitIntervalAscent} treats
natural unit interval graphs for every \(q\geq0\).
Question~\ref{ques:realRootedSpecializations} asks for the corresponding
\(q\)-refined fundamental-length result and for conditions on the weights
that explain these real-rooted specializations.

\section*{Declaration on the use of generative AI}

During the development and preparation of this work, the authors used ChatGPT 6 Astra (OpenAI) \cite{openai2026chatgpt} for exploratory discussions, including suggestions for mathematical approaches, as well as for language and presentation improvements, and in the elaboration of the pictures. Any AI-generated mathematical suggestions were independently checked, modified when necessary, and incorporated only after verification by the authors. Generative AI was not treated as a source of mathematical authority, and the authors remain fully responsible for all results, proofs, and conclusions in the manuscript.

\end{document}